\documentclass{article}

\usepackage[utf8]{inputenc}
\usepackage{a4wide}
\usepackage{amsmath,amsthm}
\usepackage{ragged2e}
\usepackage{amsfonts}
\usepackage{amssymb}
\usepackage{exscale,times,color}
\usepackage{amssymb,  csquotes, amsmath, graphicx,setspace,url}
\usepackage{color}
\usepackage{epic}
\usepackage{eepic}
\usepackage{enumerate}
\usepackage[mathscr]{eucal}
\usepackage{url}
\usepackage{setspace, tikz}
\usetikzlibrary{calc}
\usepackage{fancyhdr}
\usepackage[yyyymmdd,hhmmss]{datetime}

\newcommand{\ten}{\,\bf}

\newcommand{\ep}{{\varepsilon}}

\usepackage[normalem]{ulem}

\newtheorem{remark}{Remark}

\newtheorem{lemma}{Lemma}

\newtheorem{theorem}{Theorem}
\newtheorem{proposition}{Proposition}

\newcommand{\mathbi}[1]{{\boldsymbol{#1}}}

\newcommand{\etal}{\textit{et al.} }
\newcommand{\cf}{\textit{cf.\,}}
\newcommand{\eg}{\textit{e.g.\,}}

\newcommand{\frc}{{\mathrm{f}}}

\newcommand{\jmp}[1]{{\,\big[ \hspace{-1.0mm}\big[\,#1\,\big]\hspace{-1.0mm}\big] }}

\newcommand{\TT}{\parallel}

\newcommand{\itLa}{\mathit{\Lambda}}

\newcommand{\tsLa}{\mathbi{\itLa}}

\newcommand{\tsLaINI}{\tsLa^{\!\tsT^0}}

\newcommand{\tsT}{\mathbi{T}}

\newcommand{\veu}{\mathbi{u}}
\newcommand{\ven}{\mathbi{n}}

\newcommand{\vetau}{{\boldsymbol\tau}}

\newcommand{\vef}{\mathbi{f}}

\newcommand{\Grav}{G}

\def\V{{\boldsymbol V}}

\def\uu{{\boldsymbol u}}
\def\vv{{\boldsymbol v}}

\newcommand{\CS}{\mathscr{C}}

\definecolor{Blue}{rgb}{0,0,0.7}
\definecolor{Red}{rgb}{0.7,0,0}

\begin{document}
\title{Coupling of flow, {contact mechanics and friction, generating waves} in a fractured porous medium}

\author{Maarten V. de Hoop$^{1}$, Kundan Kumar$^2$\\[1.0em]
{$^1$}  Simons Chair in Computational and Applied Mathematics and Earth Science, \\Rice University, Houston TX, USA \\
{$^2$} Center for Modeling of Coupled Subsurface Dynamics, \\Department of Mathematics, University of Bergen, Norway\\
%{$^2$}  Universite Pierre et Marie Curie, Paris VI, Laboratoire Jacques-Louis Lions \\
%almanitm@utexas.edu, kundan.kumar@uib.no, \{tameem,mfw\}@ices.utexas.edu}
mvd2@rice.edu, kundan.kumar@uib.no}
%kundan.kumar@uib.no
\maketitle
%\date{\today}
%\begin{keyword}
%phase-field; pressurized fractures; iterative coupling algorithm; reservoir
%simulation
Poroelasticity;  Induced seismicity; {Well-posedness of initial value problems}; Fully dynamic Biot equations; Fractured porous media
%\end{keyword}

%\date{\today}

\begin{abstract}
We present a mixed dimensional model for a fractured poro-elasic medium including contact mechanics. The fracture is a lower dimensional surface embedded in a bulk poro-elastic matrix. The flow equation on the fracture is a Darcy type model that follows the cubic law for permeability. The bulk poro-elasticity is governed by fully dynamic Biot equations. The resulting model is a mixed dimensional type where the fracture flow on a surface is coupled to a bulk flow and geomechanics model. The particularity of the work here is in considering fully dynamic Biot equation, that is, including an inertia term, and the contact mechanics including friction for the fracture surface. We prove the well-posedness of the continuous model.
% \textcolor{red}{Last part of the paper where we shpow the well-posedness including friction-- this is where it seems too rough as well as too rushed. We should look at the completion of the arguments. And question: is it rigorous enough? This is too condensed for now. Regularization- what happens to the model when the regularization parameter going to zero. }
\end{abstract}

\section{Introduction}

Subsurface resource management involves anthropogenic activities involving injection and/or extraction of fluids below the earth surface. The resulting coupling of flow and geomechanics finds applications in subsidence events, ground water remediation, hydrogen energy storage, enhanced geothermal systems, solid waste disposal, and hydraulic fracturing,  (see e.g.\cite{jha2014coupled} and the references within, also see ~\cite{Dean:2009}). In geological carbon capture, storage and utilization (CCSU)  supercritical CO2 is injected into  deep subsurface reservoirs \cite{birkholzer2009large}.   In all these processes, anthropogenic activities involve injecting or extracting fluids in the subsurface leading to an imbalance in the fluid content in the subsurface. Injection may also lead to building up of pressure near the injection sites potentially causing induced seismicity. Subsidence phenomenon due to fluid extraction or induced seismicity are possible consequences of these processes. These effects are accentuated in the presence of fracture or fault zones as it provides mechanically vulnerable regions for inducing slips, earthquakes, and mechanical instabilities. Moreover, fractures are abundantly prevalent in the subsurface and strongly influence the flow profiles. Accordingly, geological and seismological considerations motivate the study of coupling of flow and geomechanics in a fractured subsurface setting. We refer \cite{martinez2013coupled} for use of this in studying maximum injection rates of carbon sequestration. A quantitative understanding of the effect of the mechanical behavior and the fluid pore pressure is important to evaluate the risks of these undesirable phenomena of early breakthrough in petroleum extraction or subsidence due to ground water extraction or induced seismicity events due to carbon sequestration or operations in enhanced geothermal systems.

Understanding interactions between in-situ stresses, fluid pressure through migration of injection, and fracture is a challenging problem, both mathematically and computationally. The complexity involved in the description of the failure of the rock and the fracture/fault slippages poses a challenge in its mathematical description due to the nonlinearities involved. The fault surface often has dissimilar mechanical properties of the neighboring rock and may have complex non-planar geometry with small width and having drastically different hydraulic properties than the nearby matrix. The heterogeneous internal structure of the fault zone leads to a complex description of the hydraulic properties depending on the mechanical displacement and stresses.  The presence of fluid pressure may have significant effects on fault dynamics and induced seismicity \cite{sibson1973interactions, scholz2019mechanics, wibberley2005earthquake, rice2006heating, cappa2011influence}. For example, fluids can facilitate slip on faults at stress levels lower than those required under dry conditions. When fluids are present in fault zones, they can escape or accumulate depending on the internal structure and the hydraulic properties of the fault zone. If the fluid pressure in the fault slip zone increases, the effective normal stress across the fault will decrease; this effect will in turn induce slippage of the fault surface by reducing the fault strength as can been seen by the Coulomb failure criterion \cite{sibson1973interactions, scholz2019mechanics, wibberley2005earthquake, rice2006heating}. We refer to \cite{cappa2011influence} for further discussion on this.
 
The coupling of flow and geomechanics with the fracture flow models is an excellent example of multiscale and multiphysics problem. The fracture is a long and thin domain and strong dimensional disparity suggests modelling it as a lower dimensional surface.  There is a  rich literature dealing with fractures and we refer to textbooks  \cite{bear2013dynamics, helmig1997multiphase, adler2012fractured} as standard references and to \cite{berre2019flow} for a recent survey of the different approaches of fracture flow models. A small aperture to length ratio makes it difficult to resolve the fluid flow explicitly through brute force computations. This is evident as the thinner the fracture is, one requires smaller grid size to resolve it. Furthermore, transverse flow details may not be important in a thin medium. Given that the resolution of the fracture by discretizing the thin and long medium is computationally infeasible, it is natural to have an explicit representation of a fracture as a surface in 3D or a curve in 2D.  We adopt this approach in our work. This is premised on the fact that such a fracture is a dominant feature and is long enough so that they cannot be included by incorporating their impact by modifying locally the hydraulic properties of the grid cells. The variant of these approaches are known as discrete fracture matrix model (DFM), Discrete fracture network models (DFN), reduced models, or more generally as mixed-dimensional models.  This has been widely used  \cite{Bastian2000,reichenberger2006mixed, fumagalli2013numerical, formaggia2014reduced, aghili2019two} for a variety of flow models, and \cite{SingPenKuGaWhe:2014, girault2016convergence, bukac2016dimensional, bukavc2015operator, andersen2016model, brenner2018hybrid, morales2012interface} for multiphysics models including geomechanics and other advanced physical models.

The components of the multiphysics model here is the elasticity and the flow both in the porous matrix and on the fracture surface. We build the multiphysics model in a \enquote{natural} manner by building on the existing models.  The flow (pressures and fluxes) and deformations (displacements) in the poroelastic medium are modeled by standard linear single phase  fully dynamic Biot's equations and the fracture flow is modeled by a linear time dependent slightly compressible single phase Darcy flow. The separate flow equation on the fracture surface allows us to explicitly resolve the flow in the  fractures. The flow in the fractures is coupled with the reservoir flow via a {\em{leakage term}} which accounts for the jump in fluxes across the interface. Consistent with the mass balance, the mechanical coupling for the fracture flow model is obtained via a term denoting the time rate of change of width of the fracture. The fluid pressure in the fracture balances the poro-elastic stresses from the matrix side.  From the fracture side, the appearance of the leakage term as a source term couples it to the bulk flow problem; the fluid storage term having the width (which is obtained by the jump in the displacement) couples to the mechanics equation in the bulk. For the poroelastic model in the porous matrix, the fracture appears as an interface requiring interface conditions. At this interface,  the continuity of stresses and the continuity of fluid pressures are imposed. This makes the laws governing the flows and geomechanics in the fractured porous media fully coupled. We also emphasize the rather strong assumptions involved in the linearization of the model analyzed here as the original model is highly nonlinear. Study of the simplified problem however provides us a mathematically tractable problem based on which further nonlinearities can be considered. We also note that the simplified model considered here retains the essence of the physical effects involved and is a good starting point for including more complex descriptions, either from the mechanical effects or the more involved flow models.

In this work, we prove existence and uniqueness of the solution of a coupled linearized system with one fracture. This work extends the previous works in \cite{girault2015lubrication} and \cite{girault2019mixed} that  study coupled flow and mechanics in a fractured porous medium, \cite{girault2015lubrication} using a conformal setting and \cite{girault2019mixed} using a mixed setting.  The particularity in our current work is two-fold: first, we consider the fully dynamic Biot equation instead of the quasi-static Biot model; and secondly, we enrich the model of fracture behavior by accounting for the contact mechanics including friction. The use of a fully dynamic Biot model instead of the more commonly used quasi-static model provides us a natural framework to extend the models for dynamically evolving fault strength, such as slip-weakening or rate- and state-dependent friction models to simulate \textit{induced slip and seismicity}. The use of contact mechanics in the fracture behavior is necessary. In the current models for the coupled flow and mechanics in a fractured setting, the width of the fracture is an unknown and the fracture permeability is a function of the width. However, the model does not have any mechanism for ensuring that the width remains positive. Thus, the existing models allow the unphysical behavior of fracture or fault surfaces penetrating into each other. To avoid this, we consider well-established contact mechanics models, e.g., the framework of Lions and Magenes\cite{LionsMagenes} and more closely Martins and Oden \cite{martins1987existence} that introduce a resistance when the fracture surface start penetrating (or when the width becomes zero). The mathematical techniques of our work borrows  from the previous works \cite{girault2015lubrication} and \cite{girault2019mixed} that include flow and mechanics and the work of \cite{martins1987existence} that includes contact mechanics and friction.

The multiphysics in more generality  is studied in the framework of coupled  Thermal-Hydro-Mechanical-Chemical models (THMC model framework). The model described in this work does not consider important effects including chemical reactions and thermal effects. In the applications involving carbon capture or hydrogen energy storage, these effects are of considerable importance. Moreover, our flow model is a single phase model as opposed to multiphase flow in most of applications. However, the model considered here has several challenges and extensions without considering these extra mechanisms. Incorporating the effects of friction and contact mechanics in the fractured/fault case and extending this to include rate and state variables in the presence of fluids remains a question of high and active interest. We accordingly, restrict ourselves to the simplest situation that captures the key components, that is, coupling of flow, deformations but include normal and friction forces on a fractured/fault surface. %Secondly, incorporating geochemistry, thermal, and multiphase is quite complex and out of scope of the present considerations.

We provide a brief summary of the relevant literature. 
The framework of coupled flow and geomechanics has been based on the pioneering work of Terzaghi \cite{ ref:Terz43} and Biot \cite{ref:Biot41a,ref:Biot41b}. Since these foundational works, there has been active investigation into the coupled geomechanics and flow problems including nonlinear extensions \cite{ref:Coussy89,ref:LewSch98}. We refer the reader to \cite{ref:SetMo94,ref:Fung94,ref:Gai03,ref:Chin98} and the references therein for representative works dealing with these coupled problems.    The quasi-static Biot systems have  been analyzed by a number of authors who established existence, uniqueness, and regularity, see Showalter ~\cite{ref:Showt2000} and references therein,  Phillips \& Wheeler~\cite{ref:PhW1h07}. More recently, Mikeli\'c and Wheeler \cite{ref:Mikwheel12} have considered different iterative schemes for flow and mechanics couplings and established  contractive results  in suitable norms; see also  \cite{kim11,kim2} for  studying the stability of iterative algorithms. For a variational characterization of iterative schemes, we refer to the recent work \cite{both2019gradient}. A parallel in time extension of the iterative scheme was proposed in the work of \cite{br2}, and a fixed-stress split based preconditioner for the simultaneously coupled system was proposed in the work of \cite{ref:Gai03,  white2016block, gaspar2017fixed}. There has been an investigation into the optimization of fixed stress algorithm in \cite{storvik2019optimization}. Multiphase flow coupled with linear/nonlinear geomechanics has been solved using iterative algorithm and its variants \cite{both2019anderson, murad2021fixed, quevedo2021novel}. In the context of fractured poroelastic medium, there have been several approaches for numerical computations \cite{franceschini2020algebraically, mehmani2021multiscale, vuik2021preconditioned, stefansson2020numerical}. In a closely related work, \cite{boon2023mixed} uses  mixed dimensional calculus tools to study coupled flow and geomechanics in a fractured medium. \cite{deHoop2020, ye2020multi} studies the elastic wave equations including friction behavior of the fault surface without involving any fluid. Recently, \cite{IvanFracture} shows the stability of a numerical scheme based on Nitsche's method for a Stokes-Biot model where the flow in the fracture is modeled by the Stokes equation.   

The paper is structured as follows. In Section \ref{sec:notation}, we discuss the notation and functional spaces used. This is followed by model equations in Section \ref{sec:modeleqns}. A weak formulation is developed in Section \ref{sec:weak} and analyzed in Section \ref{sec:wellposedness} for well-posedness. Contact mechanics including the friction effects has been incorporated in Section \ref{sec:contact2}.

\section{Notation and functional setting} \label{sec:notation}

We denote scalar quantities in italics, e.g., $\lambda, \mu, \nu$, a typical point in $\mathbb{R}^d, d = 2 \text{ or } 3$ and vectors in bold e.g., $\it{\bf{x}, \bf{y}, \bf{z}}$, matrices and second order tensors in bold face, e.g., $\textbf{A}, \textbf{B}, \textbf{C}$, and fourth order tensors in blackboard face, e.g., $\mathbb{A}, \mathbb{B}, \mathbb{C}$. The symmetric part of a second order tensor $\textbf{A}$ is denoted by $\hat{\ten{A}} = \frac{1}{2} (\ten{A}+\ten{A}^T)$ where $\ten{A}^T$ is the transpose of $\ten{A}$. In particular, $\widehat{\nabla \boldsymbol{u}}$ represents the deformation tensor. An index that follows a comma indicates a partial derivative with respect to the corresponding component of the spatial variable $\bf x$; we also use the $\nabla$ to denote the gradient operator.  For integrals, we use $\mathrm{d}x$  for the volume element and $\mathrm{d}a$ for the surface element. For a time dependent quantity, a dot above a variable or $\partial_t$ represents the time derivative of the variable. 

\par Our interest is in mathematical models describing the evolution or the equilibrium of the mechanical state of a body within the framework of linearized strain theory. We use the symbols ${\uu}, \boldsymbol{\sigma}, \text{ and } \boldsymbol{\varepsilon} = \boldsymbol \varepsilon(\uu)$ for the displacement vector, the stress tensor, and the linearized strain tensor. The components of the strain tensor are given by
\begin{align*}
{\boldsymbol{\varepsilon}}({\uu}) = ({\boldsymbol{\varepsilon}}({\uu}))_{ij} = \dfrac{1}{2} ({ u}_{i,j} + { u}_{j,i}).
\end{align*}
Here $u_{i,j} = \dfrac{\partial u_i} {\partial x_j}$. 

We denote the space of second order symmetric tensors on $\mathbb{R}^d$ by $\mathbb{S}^d$. We use standard notation for inner products, that is, 
\begin{align*}
{\boldsymbol{u}}\cdot {\boldsymbol{v}} = \sum_i u_i v_i, \|{\boldsymbol v}\| = ({\boldsymbol v} \cdot {\boldsymbol v})^{1/2} \text{ for all } {\boldsymbol u} = (u_i), {\boldsymbol v} = (v_i) \in \mathbb{R}^d, \\
 {\boldsymbol{\sigma}}:{\boldsymbol{\tau}} = \sum_{i,j} \sigma_{ij}\tau_{ij}, \|{\boldsymbol{\tau}}\| = ({\boldsymbol{\tau}}, {\boldsymbol{\tau}})^{1/2} \text{ for all } {\boldsymbol{\sigma}} = (\boldsymbol \tau_{ij}) \in \mathbb{S}^d,
\end{align*}
respectively. We use $\boldsymbol{0}$ for the zero element of the spaces $\mathbb{R}^d, \mathbb{S}^d$. Here, the indices $i,j$ run between $1$ and $d$. Unless otherwise indicated, the summation convention over repeated indices is implied. 

% We represent a ball of given radius $r$ centered at $\bf{x}$ by $B_r(\bf x) \subset \mathbb{R}^3$ and a circle of radius $r$ and center $\bf y$ by $B_r(\bf{y})^\prime \subset \mathbb{R}^2$.

% We let $\Omega$ be the reference configuration of the body, assumed to be open, bounded, connected set in $\mathbb{R}^d$ with a Lipschitz boundary $ \partial \Omega$.  {\color{red} [introduce notation for $\Gamma$ and $\Gamma_{\pm}$, here; where is $\Gamma$ used?]}

Let the reservoir $\Omega$ be a bounded  domain of $\mathbb{R}^d$  $d=2$ or $3$, with a piecewise smooth Lipschitz boundary $\partial\Omega$ and exterior normal $\boldsymbol{n}$ and let $\Gamma$ be an open subset of $\partial\Omega$ with positive measure. Let the fracture $\CS \Subset \Omega$ be a simple closed piecewise smooth curve with endpoints $\gamma_1$ and $\gamma_2$ when $d=2$ or a simple closed piecewise smooth surface with piecewise smooth Lipschitz boundary $\partial \CS$ when $d = 3$. The reservoir contains both the matrix and the fractures; thus the reservoir matrix is $\Omega \setminus \CS$. The closure of the domain will be denoted by $\overline{\Omega}$.  We denote by $\boldsymbol \tau^\star_j$, $1 \le j \le d-1$, a set of orthonormal tangent vectors on $\CS^\star$,  $\star = +,-$. The tangential derivative on the surface $\CS$ will be denoted by $\overline \nabla$.

\subsection*{Function spaces}

We use standard notations for Sobolev and Lebesgue spaces over $\Omega$ or on $\Gamma$. In particular, We denote by $L^2(\Omega)$ the space of square integrable functions on a Lipschitz  domain $\Omega$ or $L^2(\Omega; \mathbb{R}^d)$, $L^2(\Gamma; \mathbb{R}^d)$, $H^1(\Omega; \mathbb{R}^d)$ endowed with their  canonical inner products and associated norms. For an element ${\boldsymbol v} \in H^1(\Omega; \mathbb{R}^d)$ we write ${\boldsymbol v}$  for its trace $\gamma {\boldsymbol v}$ on $\Gamma$. Let $\Gamma \subset \partial \Omega$ be an open set at which homogeneous Dirichlet boundary data is prescribed. 
% Define the function spaces 
% \begin{equation*}
% V = \{{\bf v} \in H^1(\Omega; \mathbb{R}^d) ; {\bf v} = 0 \text{ on } \Gamma_1\} ,\quad
% Q = \{{\ten{\sigma}} = (\sigma_{ij}) \in L^2(\Omega) ; \sigma_{ij} = \sigma_{ji}\}
% \end{equation*}
% {\color{red} [in the weak formulation, $\boldsymbol{V}$ is used as notation]} for the displacement and the stress, respectively. These are real Hilbert spaces endowed with the inner products {\color{red} [inner product for $Q$ missing]}
% \begin{align*}
% ({\bf u}, {\bf v})_V = \int_\Omega {\ten{\ep({\bf u})}} \cdot  {\ten{\ep({\bf v})}} {\rm{d}}x,
% \quad
% ({\ten u}, {\bf v})_V = \int_\Omega {\ten{\ep({\bf u})}} \cdot  {\ten{\ep({\bf v})}} {\rm{d}}x,
%\end{align*}

We define the function spaces 
 \begin{equation*}
 H^1_{0,\Gamma_1}(\Omega) = \{{\boldsymbol v} \in H^1(\Omega; \mathbb{R}^d) ; {\boldsymbol v} = 0 \text{ on } \Gamma_1\},
 %Q = \{{\ten{\sigma}} = (\sigma_{ij}) \in L^2(\Omega) ; \sigma_{ij} = \sigma_{ji}\}
 \end{equation*}
We list Poincar\'e's, Korn's, and some trace inequalities that will be used later. 
Poincar\'e's inequality in $H^1_{0,\Gamma_1}(\Omega)$ reads: There exists a 
constant ${\mathcal P}_{\Gamma_1}$ depending only on $\Omega$ and $\Gamma_1$ such that
\begin{equation}
\forall v \in H^1_{0,\Gamma_1}(\Omega)\,,\, \|v\|_{L^2(\Omega)} \le {\mathcal P}_{\Gamma_1} |v|_{H^1(\Omega)}.
\label{eq:Poincare}
\end{equation}
Here, $|v|_{H^1(\Omega)}$ is the $H^1$ semi-norm.  Next, we recall Korn's first inequality in $H^1_{0,\Gamma_1}(\Omega)^d$: There exists a constant $C_\kappa$ depending only on $\Omega$ and $\Gamma_1$ such that
\begin{equation}
\forall \vv \in H^1_{0,\Gamma_1}(\Omega)^d\,,\, |\vv|_{H^1(\Omega)} \leq C_{\kappa} \|\varepsilon(\vv)\|_{L^2(\Omega)}.
\label{eq:Korn1}
\end{equation}  
We shall use the following trace inequality in $H^1(\Omega)$: There exists a constant $C_\tau$ depending only on $\Omega$ and $\partial \Omega$ such that
\begin{equation}
\forall \varepsilon >0\,,\,\forall v \in H^1(\Omega)\,,\, \|v\|_{L^2(\partial \Omega)} \leq 
\varepsilon\|\nabla\,v\|_{L^2(\Omega)} + \big(\frac{C_\tau}{\varepsilon} + \varepsilon\big)\|v\|_{L^2(\Omega)}.
\label{eq:trace}
\end{equation}
This inequality follows for instance from the interpolation inequality (see Brenner \& Scott~\cite{ref:BrSc})
$$ \forall v \in H^1(\Omega)\,,\, \|v\|_{L^2(\partial \Omega)} \leq C \|v\|^{\frac{1}{2}}_{L^2(\Omega)} \|v\|^{\frac{1}{2}}_{H^1(\Omega)},$$
and Young's inequality. Besides \eqref{eq:trace}, by combining \eqref{eq:Poincare} and \eqref{eq:Korn1}, we immediately derive the alternate trace inequality, with a constant $C_D$ depending only on $\Omega$ and $\partial \Omega$:
\begin{equation}
\forall \vv \in H^1_{0,\Gamma_1}(\Omega)^d\,,\, \|\vv\|_{L^2(\partial \Omega)} \le C_D \|\ep(\vv)\|_{L^2(\Omega)}.
\label{eq:trace2}
\end{equation}

% As far as the divergence operator is concerned, we shall use the space 
% $$H({\rm div};\Om) = \{\vv \in L^2(\Om)^d\,;\, \nabla\cdot\vv \in L^2(\Om)\},$$
% equipped with the norm
% $$\|\vv\|_{H({\rm div};\Om)} = \big(\|\vv\|^2_{L^2(\Om)} + \|\nabla\cdot\vv\|^2_{L^2(\Om)}\big)^{\frac{1}{2}}.$$

As usual, for handling time-dependent problems, it is convenient
to consider functions defined on a time interval $]a,b[$ with
values in a functional space, say $X$ (cf.~\cite{ref:LM72}). More precisely, let $\|\cdot\|_X$ denote
the norm of $X$; then for any $r$, $1 \le r \le \infty$, we
define
$$
L^r(a,b;X) = \left\{f \mbox{ measurable in } ]a,b[\,; \int_a^b \|f(t)\|_X^rdt <\infty\right\},
$$
equipped with the norm
$$\|f\|_{L^r(a,b;X)} = \left(\int_a^b \|f(t)\|_X^rdt\right)^{\frac{1}{r}},$$
with the usual modification if $r = \infty$. This space is a Banach space
if $X$ is a Banach space, and for $r=2$, it is a Hilbert space if $X$ is a Hilbert space. To simplify, we sometimes denote derivatives with respect to time with a prime and we define for any $r$, $1 \le r \le \infty$,
$$W^{1,r}(a,b;X) = \{f \in L^r(a,b;X)\,;\, f^\prime \in L^r(a,b;X)\}.$$
For any $r\ge 1$, as the functions of $W^{1,r}(a,b;X)$ are continuous with respect to time, we define
$$W^{1,r}_0(a,b;X) = \{f \in W^{1,r}(a,b;X)\,;\, f(a) = f(b) =0\},$$
and we denote by $W^{-1,r}(a,b;X)$ the dual space of $W^{1,r^\prime}_0(a,b;X)$,  where $r^\prime$ is the dual exponent of $r$, ${\frac{1}{r^\prime}}+{\frac{1}{r}} =1$. 

The spaces for our unknowns are described below. To simplify the notation, the spaces related to $\CS$ are written $L^2(\CS)$, $H^{1}(\CS)$, etc, although they are defined in the interior of $\CS$.  It is convenient  to  introduce an auxiliary partition of $\Omega$ into two non-overlapping subdomains $\Omega^+$ and $\Omega^-$ with Lipschitz interface $\Gamma$ containing $\CS$, $\Omega^\star$ being adjacent to $\CS^\star$, $\star = +, -$. The precise shape of $\Gamma$ is not important as long as $\Omega^+$ and $\Omega^-$ are both Lipschitz. Let $\Gamma^\star= \partial \Omega^\star\setminus \Gamma$. For any function $f$ defined in $\Omega$, we extend the star notation to $\Omega^\star$ and set $f^\star = f_{|\Omega^\star}$, $\star = +, -$.
Let $W = H^1(\Omega^+ \cup \Omega^-)$ with norm
$$\|v\|_W = \big(\|v^+\|^2_{H^1(\Omega^+)} + \|v^-\|^2_{H^1(\Omega^-)}\big)^{\frac{1}{2}}.$$
The space for the displacement is $\uu \in L^\infty(0,T;\V), \partial_t \uu \in L^\infty(0,T;L^2(\Omega)) \displaystyle \cap L^2(0,T;\V), \partial_{tt} \uu \in L^2(0,T;\V^\prime)$, where $\V$ is a closed subspace of $H^1(\Omega\setminus \CS)^d$:
\begin{equation}
\V = \{\vv \in W^d\,;\,  [\vv]_{\Gamma\setminus \CS} = {\bf 0}, \vv^\star_{|\Gamma^\star} = {\bf 0},\star = +,-\},
\label{eq:V}
\end{equation}
with the norm of $W^d$:
\begin{equation}
\|\vv\|_\V = \big(\sum_{i=1}^d \|v_i\|^2_W\big)^{\frac{1}{2}}.
\label{eq:normV}
\end{equation}
% \noindent We define $W = H^1(\Omega^+ \cup \Omega^-)$, that is, 
% \[
% W = \{v \in L^2(\Omega); v^+ \in H^1(\Omega^+) ,\ v^- \in H^1(\Omega^-) \} 
% \]
% normed by the graph norm
% \[
% \|v\|_W = \left(\|v^+\|^2_{H^1(\Omega^+)}+\|v^-\|^2_{H^1(\Omega^-)} \right)^{1/2}.
% \]
 % {\color{red} [$V$: ambiguous notation]} 
%\[ V = \{\boldsymbol v \in W^d; [\boldsymbol v]_{\Gamma\setminus \CS} = {\bf{0}},\ {\boldsymbol v}^{+}|_{\Gamma^+} {\color{red} = ..} ,\ {\boldsymbol v}^{-}|_{\Gamma^-} = {\bf 0}  \} \]
% with the norm defined on $W^d$
% \[
% \|\boldsymbol v\|_V = \left ( \sum_{i = 1}^d  \| v_i \|_W^2 \right )^{1/2}. 
% \]
For the space for displacement, the regularity in space will be studied in $\V$ as stated above. By the continuity of the trace operator on $\Gamma$ both in $H^1(\Omega^+)$ and $H^1(\Omega^-)$, it follows that $V$ is a closed subspace of $W^d$. Moreover, the countable basis of $W$ implies separability of $V$.  \par In order to specify the relationship between the displacement $\uu$ of the medium and the width $w$ of the fracture, we distinguish the two sides of $\CS$ by the superscript $* = +, -$ with $\Omega^*$.

We state the following result. The proof can be found in the appendix of \cite{girault2015lubrication}.\\

\begin{proposition}
The space $\boldsymbol V$ is a closed subspace of $H^1(\Omega \setminus \CS)^d$.
\end{proposition}

 There is a non-negative function $w$ assumed to be $H^1$ in time and smooth in space away from the crack front. It vanishes on $\partial \CS$ and in the neighborhood of $\partial \CS$, $w$ is asymptotically of the form
 \begin{align}
\label{eq:hyp_w} w(x,y) \approx x^{\frac{1}{2}+\epsilon}f(y), \text{ with small } \epsilon > 0,     
 \end{align}
where $y$ is locally parallel to the crack front, $x$ is the distance to the crack, and $f$ is smooth. 
The width in he fracture goes to zero at the boundaries of the fracture. To take care of this degeneracy, we need to define weighted spaces,
\[H^1_w(\CS) = \{z \in H^{1/2}(\CS); w^{\frac 3 2} \bar \nabla z \in L^2(\CS)^{d-1}\}. \]
The norm on the space $H^1_w$ is given by
\[\|z\|^2_{H^1_w(\CS)} =  \|z\|^2_{H^{1/2}(\CS)} + \|w^{\frac 3 2}\bar{\nabla} z\|_{L^2(\CS)}^2 .  \]
We collect some facts about the space $H^1_w(\CS)$, the proofs of which can be found in  \cite{girault2015lubrication}. The asymptotic form of $w$ guarantees the following result

\begin{theorem} \label{th:weightedspace}
$H_w^1(\CS)$ is a Hilbert space, $W^{1,\infty}(\CS)$ is dense in $H_w^1(\CS)$, $H^1_w(\CS)$ is separable, and the following Green's formula holds for all $\theta \in H_w^1(\CS)$ such that 
$\overline{\nabla} \cdot (w^3 \overline{\nabla} \theta) \in H^{-\frac{1}{2}}(\CS)$, 
\[\text{ for all } v \in H^1_w(\CS),\quad -\langle \overline{\nabla} \cdot (w^3 \overline \nabla \theta), v \rangle_{\CS} = \int_{\CS} w^{\frac 3 2} \overline{\nabla} \theta \cdot w^{\frac 3 2} \overline{\nabla} v dx. \]
%{\color{red} [I suggest writing the measure; $\overline{\nabla}$ needs to be defined and the notation ($\lambda$) is ambiguous]}
\end{theorem}

\noindent We define the mixed space where pressure, $p$, will be defined
\[
Q = \{p \in H^1(\Omega); p_c \in H_w^1(\CS) \}.
\]
Here, $p_c$ is the trace of $p$ which exists by virtue of $p$ being in $H^1(\Omega)$. The norm on this mixed space is given by
\[
\|p\|_Q^2 = \|p\|_{H^1(\Omega)}^2+ \|p_c\|^2_{H^{1/2}(\CS)} + \|w^{\frac 3 2} \overline{\nabla} p_c\|_{L^2(\CS)}^2.
\]
Using that $H_w^1(\CS)$ is a separable Hilbert space, the following is immediate

\begin{theorem}
$Q$ is a separable Hilbert space. Moreover, there exists a continuous linear extension operator, $E :\ H_w^1(\CS) \to Q$.
\end{theorem}

% \noindent
% Furthermore, we have

% \begin{theorem}

% \end{theorem}

\section{Biot--Allard model in a fractured porous medium}
\label{sec:modeleqns}

We describe the coupling of flow and mechanics in a fractured subsurface. The model equations are of mixed dimensional type where a 3D porous matrix model is coupled to a 2D fracture model. 
Recall the subsurface domain $\Omega$  with a piecewise smooth Lipschitz boundary $\partial \Omega$ and exterior normal $\boldsymbol n$ and the fracture $\mathscr C \Subset \Omega$  with piecewise smooth Lipschitz boundary $\partial \mathscr C$. See the schematic in Figure \ref{fig:fracture} below.  

\subsection{Elasticity equation in $\Omega\setminus \mathscr C$}
\label{subsec:interior}

The displacement of the solid is modeled in $\Omega\setminus \mathscr C$ by the dynamic Biot equations for a linear elastic, homogeneous, isotropic, porous solid saturated with a slightly compressible viscous fluid.  

\begin{figure}
\centering
\begin{tikzpicture}
    % Draw the rectangle
    \draw [line width=3.5pt] (0,0) rectangle (12,8);

    % Define the curve
    \draw [line width=3.5pt] plot [smooth] coordinates {(2,3.2) (2.5,3.4)  (3,3.8) (4,4) (5,4.2) (5.6,4.4) (6.2, 4.6) (9,4.7)};
 %   \draw[blue, thick] (5,4) .. controls (6,6) and (7,3.5) .. (7.5,4.5);

    % Define the normal vector
    \coordinate (P) at (5,4.2); % Coordinates of the point on the curve
    \draw[line width=2.5pt, ->] (6.2, 4.6) -- ++(0, -1) node[below, font=\Large] {$\boldsymbol n^{+}$};
    %\coordinate (N) at ($ (P) + (1.5, 0.5) $); % Endpoint of the normal vector

    % Draw the normal vector
    %\draw[red, line width=3.5pt, ->] (P) -- (N);

    % Label the point and the normal vector
   % \node[above right] at (P) {$P$};
    %\node[right] at (N) {$\vec{N}$};

    % Define the offset for parallel lines
   % \def\offset{0.2}

    % Draw the upper parallel line
    % \draw[green, thick] (1-\offset,1+\offset) .. controls (2-\offset,3+\offset) and (3-\offset,0.5+\offset) .. (3.5-\offset,1.5+\offset);

    % % Draw the lower parallel line
    % \draw[orange, thick] (1+\offset,1-\offset) .. controls (2+\offset,3-\offset) and (3+\offset,0.5-\offset) .. (3.5+\offset,1.5-\offset);

    % Label the original curve
    \node[above right, black, font=\Large] at (5.1,4.3) {$\mathscr{C}$};
     \node[above right, black, font=\Large] at (10.1,4.8) {$\Gamma$};
     \node[above right, black, font=\Large] at (6,5.4) {$\Omega^{+}$};
      \node[above right, black, font=\Large] at (7.5,2.5) {$\Omega^{-}$};

    % Draw dotted lines connecting curve endpoints to the rectangle boundary
    \draw[dotted, line width=2.5pt] (2,3.2) -- (0,3);
    \draw[dotted, line width=2.5pt] (9,4.7) -- (12,4.8);
\end{tikzpicture}
 \caption{Schematic of a fracture $\mathcal{C}$ embedded in a rectangular subsurface medium. $\Gamma$ is the extension of $\mathscr{C}$ dividing the domain $\Omega$ into two subdomains $\Omega^{+}$ and $\Omega^{-}$. An outward normal to the boundary of subdomain $\Omega^{+}$ at the surface $\mathscr{C}$ is $\boldsymbol{n}^{+}$.}
  \label{fig:fracture}
 \end{figure}
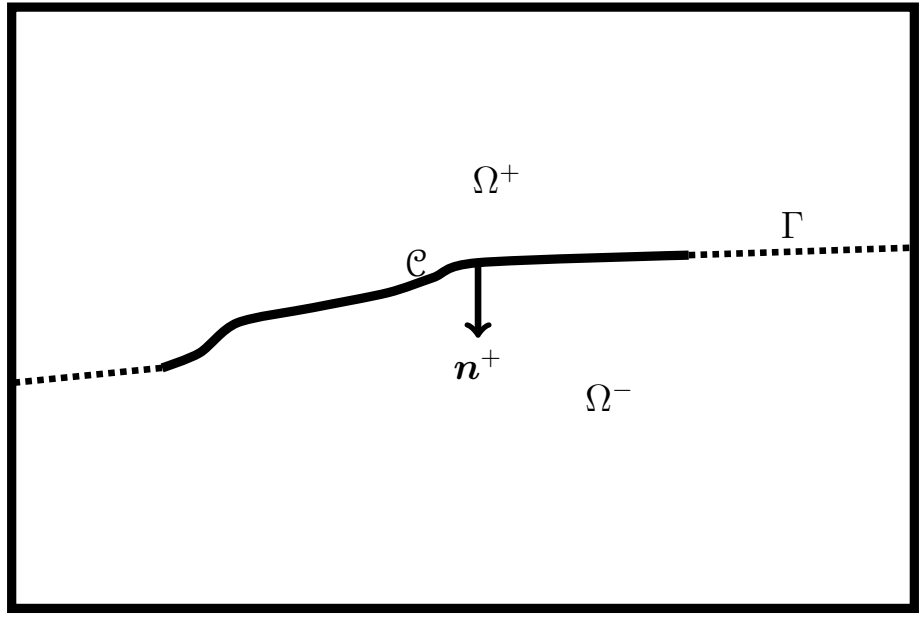

The constitutive  equation for the Cauchy stress tensor $\boldsymbol \sigma^{\rm por}$ is
\begin{equation}
\boldsymbol \sigma^{\rm por}(\boldsymbol u,p)  = \boldsymbol \sigma(\boldsymbol u) -\alpha\, p\, \boldsymbol I
\label{eq:Stressporo}
\end{equation}
where $\boldsymbol I$ is the identity tensor, $\boldsymbol u$ is the solid's displacement, $p$ is the fluid pressure, and $\boldsymbol \sigma$ is the effective linear elastic stress tensor,
\begin{equation}
\boldsymbol \sigma(\boldsymbol u) = 2G \boldsymbol{\varepsilon}(\uu) + \lambda (\nabla \cdot \uu) \boldsymbol I . %\mathbb{C} \nabla \boldsymbol u
\label{eq:Stresselas}
\end{equation}
Here, $G$ and  $\lambda$ are Lame parameters and $\alpha > 0$ is the dimensionless Biot coefficient.
Then the balance of linear momentum in the solid reads 
\begin{equation}
\partial_{tt} \uu -\rm{div} \,\boldsymbol \sigma^{\rm por}(\boldsymbol u,p) = \boldsymbol{f} \quad \text{in } \, \Omega\setminus \mathscr C,
%\rho_b g \nabla\,h,
\label{eq:balance}
\end{equation}
where $\boldsymbol f$ is a body force. In our model this will be a gravity loading term. This can be generalized to include other forces including tidal forces. In general, $\boldsymbol f$ may come from solving a potential equation (see Remark \ref{rem:ye_paper} below). 

For the case that includes friction, we need to add a viscous regularization in the elasticity equation. This is necessitated by the presence of a time derivative of displacement on the fracture surface equation requiring higher order estimates on the spatial derivative of the velocity. Accordingly, we present the regularized version of the elasticity equation. The stress is given by (with $\gamma>0$ a constant)

\begin{equation}
\boldsymbol \sigma(\boldsymbol u) = 2G \boldsymbol{\varepsilon}(\uu) + \lambda (\nabla \cdot \uu) \boldsymbol I +  \gamma \boldsymbol{\varepsilon}(\partial_t \uu) + \gamma (\nabla \cdot \partial_t \uu) \boldsymbol I. %\mathbb{C} \nabla \boldsymbol u
\label{eq:Stresselas2}
\end{equation}

We write the mechanics equation in terms of displacement $\uu$ as an unknown  by substituting \eqref{eq:Stresselas2} into \eqref{eq:Stressporo} and then substitute \eqref{eq:Stressporo} into the momentum balance equation \eqref{eq:balance}  
\begin{equation}
\partial_{tt} \uu - \gamma \nabla \cdot \boldsymbol{\varepsilon}(\partial_t \uu) - \gamma \nabla \cdot (\partial_t \uu \;  \boldsymbol I) - \nabla \cdot  \left( 2G\boldsymbol{\varepsilon}(\uu) + \lambda (\nabla \cdot \uu) \boldsymbol I -p\boldsymbol I\right) = \boldsymbol{f} \quad \text{in } \; \, \Omega \setminus \mathscr C.
%\rho_b g \nabla\,h,
\label{eq:bulk_mechanics}
\end{equation}

\begin{remark}
    Extension to heterogeneous and anisotropic elasticity coefficients with standard assumptions of strong ellipticity and bounded is quite straightforward as our results can be easily generalized. In this case, the Lame coefficients are replaced by a fourth order tensor $A_{ijkl}$ and the viscoelastic regularization parameter $\gamma$ is replaced by another fourth order tensor $C_{ijkl}$. This is supplemented with the standard conditions \cite{martins1987existence}: 
    \begin{align*}
        A_{ijkl}, C_{ijkl} \in L^\infty(\Omega \setminus \mathscr C), \\
        A_{ijkl} (\boldsymbol x) = A_{jikl} (\boldsymbol x) = A_{ijlk} (\boldsymbol x) = A_{klij} (\boldsymbol x) \quad  \text{ a.e. } \boldsymbol x \in \Omega \setminus \mathscr C \\
        C_{ijkl} \quad (\boldsymbol x) = C_{jikl} (\boldsymbol x) = C_{ijlk} (\boldsymbol x) = C_{klij} (\boldsymbol x) \text{ a.e. } \boldsymbol x \in \Omega \setminus \mathscr C\\
        \exists \alpha_A, \alpha_C > 0, A_{ijkl}A_{kl}A_{ij} \geq \alpha_A A_{ij}A_{ij}, C_{ijkl}A_{kl}A_{ij} \geq \alpha_C A_{ij}A_{ij},  \quad  \text{ a.e. } \boldsymbol x \in \Omega \setminus \mathscr C,
            \end{align*}
            for every symmetric matrix $A_{ij}$. 
\end{remark}
In the first part we discuss the linear case, ignoring contact mechanics, described by a purely linear elastic case: $\boldsymbol \sigma$ given by \eqref{eq:Stresselas}. In the second part, where we include the friction effects, we consider the linear viscoelastic model \eqref{eq:bulk_mechanics} in which case the elastic stress $\boldsymbol \sigma$ satisfies \eqref{eq:Stresselas2}. 
\subsection{Flow equation in $\Omega\setminus \CS$}
The velocity of the fluid $\boldsymbol v^D$ in $\Omega\setminus \mathscr C$ obeys Darcy's law:
\begin{equation}
\boldsymbol v^D = -\frac{1}{\mu_f}\boldsymbol K\big(\nabla\,p - \rho_f g \nabla\,\eta\big),
\label{eq:Darcy}
\end{equation}
where $\boldsymbol K$ is the absolute permeability and is a second order positive definite tensor, $\mu_f >0$ is the constant fluid viscosity, $g$ is the gravitation constant, and $\eta$ is the distance in the vertical direction, variable in space, but constant in time. The gravitational force acts in the negative $z$-direction and the flow moves in the direction of the negative of the pressure gradient superposed with the gravitational field.  

%To simplify, we also suppose that $\K$ is continuous.
 We let $\varphi^*$ denote the fluid content  of the medium. In a linear regime, we take
 \[\varphi^* = \varphi_0 + \alpha \nabla \cdot \uu + \frac{1}{M}. \]
 The first term on the right is the default porosity and the second term denotes the change in porosity due to elasticity mechanical effects.
 The fluid mass balance in $\Omega \setminus \CS$ reads
\begin{equation}
{\partial_t}\big(\rho_f \varphi^*\big)  + \nabla\cdot(\rho_f  \boldsymbol v^D) = q,
\label{eq:mbal}
\end{equation}
where $q$ is a mass source or sink term taking into account injection into or  leaking out of the reservoir. We will linearize the above equation in the following.  From \eqref{eq:mbal}
%we invoke the following approximation: 
%$\varphi^* \approx \varphi_0$  and $ \rho_f \approx \rho_{f,r} \exp(c_f(p-p_{r}))$ for the first term,

\begin{equation*}
\varphi^*  {\partial_t \rho_f}  +  \rho_f {\partial_t \varphi^*}   + \nabla\cdot(\rho_f  \boldsymbol v^D) = q.
\end{equation*}

 Let us neglect small quantities by means of the following approximations: 
\begin{align*}
&\frac{1}{M}(1+ c_f (p-p_r)) \approx \frac{1}{M}\ ,\ c_f\big(\varphi_0 + \alpha \nabla \cdot \uu +
\frac{1}{M} p\big) \approx c_f \varphi_0\ , \ \rho_{f,r} (1 + c_f (p-p_r)) \alpha \approx \rho_{f,r} \alpha,\\
& \rho_{f,r} (1 + c_f (p-p_r)) \vv^D \approx \rho_{f,r}\vv^D \ , \  \rho_{f,r} (1 + c_f (p-p_r))g \nabla\,\eta \approx \rho_{f,r}g \nabla\,\eta.
 \end{align*}

This gives a linear equation,  
% \begin{equation*}
% \varphi_0  \rho_{f,r} c_f {\partial_t p}  +  \rho_{f,r} \alpha \nabla \cdot {\partial_t \uu}   + \nabla\cdot(\rho_{f,r}  \boldsymbol v^D) = q,
% \end{equation*}
% whereby dividing by $\rho_{f,r}$ leads to
% \begin{equation}
% \varphi_0  c_f {\partial_t p}  +   \alpha \nabla \cdot {\partial_t \uu}   + \nabla\cdot \boldsymbol v^D = \tilde{q}.
% \label{eq:mbal3}
% \end{equation}
% Here $\tilde{q} = \frac{q}{\rho_{f,r}}$. The above computations linearizes the flow equation. 
% The flow equation in terms of pressure as unknown is obtained by substituting flux expression \eqref{eq:Darcy} into \eqref{eq:mbal3}, 

\begin{equation}
\varphi_0  \left(c_f + \frac{1}{M}\right)  {\partial_t p}  +   \alpha \nabla \cdot {\partial_t \uu}   - \nabla\cdot \dfrac{\boldsymbol K}{\mu_f} \left(\nabla p - \rho_{f,r} g \nabla \eta \right) = {q}
\label{eq:bulk_flow}
\end{equation}
where we have already substituted flux expression \eqref{eq:Darcy} in the above equation and the source term is suitably rescaled as well. 

We note that the mechanics equation \eqref{eq:bulk_mechanics} and the flow equation \eqref{eq:bulk_flow} form a closed system. In case of the absence of fracture, they form a coupled flow and mechanics system which is referred as the fully dynamic Biot equations. 

\begin{remark}
The above equations for the displacement are a simplified form of the more general dynamic Biot equations. These equations are obtained from the linearized first principles fluid/structure pore level interaction, using a homogenization approach \cite{mikelic2012theory}. They are obtained by making several assumptions including dropping the inertial term in the Navier-Stokes equations, supposing an incompressible or a slightly compressible pore fluid, using a linear elastic model (Navier’s equations) to describe the solid skeleton and linearization of the fluid/solid interface coupling conditions. The homogenization gives convolution terms for permeability and other higher order terms of displacement which are ignored in the above model. For further details, we refer to \cite{mikelic2012theory}. 
 \end{remark}

\subsection{Equation in the fracture surface, $\mathscr C$}
\label{subsec:fract}

The width $w$ of the fracture  is the jump of displacement, $\uu$, on the fracture surface, that is, 
 $$
w = - [\boldsymbol u]_\mathscr C \cdot \boldsymbol n^+ .
$$
with $\boldsymbol {n}^+ = - \boldsymbol {n}^{-}$ being the unit normal vectors to $\mathscr{C}$ exterior to the two sides $\Omega^{+}$ and $\Omega^{-}$ dividing $\Omega$ by the fracture surface $\mathscr{C}$. 

 In analogy with the Darcy flow in bulk,  the volumetric flow rate $\boldsymbol Q$ on $\mathscr C$ satisfies
$$
\boldsymbol Q = -\frac{w^3}{ 12\mu_f}\big(\overline{\nabla}\,p_c - \rho_f g \overline{\nabla}\,\eta\big),
$$
with $\frac{w^3}{12}$ being the permeability for the fracture flow since the permeability is a function of width. We recall that $\overline \nabla$ denotes the tangential derivative on the surface $\CS$. The conservation of mass in the fracture reads,
\begin{align}
{\partial_t} (\rho_f w) + \overline{\nabla} \cdot (\rho_f \boldsymbol Q) = q_W - q_L,
\label{eq:mbalc}
\end{align}
where $q_W$ is a known injection term into the fracture and $q_L$ is an unknown leakoff term from the fracture into the reservoir that guarantees the conservation of mass in the system. Here, $p_c$ is the pressure of the fluid on the fracture surface. We linearize the equation as in the case of bulk flow equation.  Expanding the first term and then using the approximation $\rho_f \approx \rho_{f,r}$ in the first term and the flux term,  $ \rho_f \approx \rho_{f,r} \exp(c_f(p-p_{r}))$ in the second term on the left hand side, 
\begin{align*}
\rho_{f,r} {\partial_t w}  + w c_f {\partial_t p}  + \overline{\nabla} \cdot (\rho_f \boldsymbol Q) = q_W - q_L.
\end{align*}
Dividing by $\rho_{f,r}$ throughout the equation, and redefining $c_{fc} = \frac{w c_f}{\rho_{f,r}}, \tilde{q}_W = \frac{q_W}{\rho_{f,r}}, \tilde{q}_L = \frac{q_L}{\rho_{f,r}}$ we obtain
\begin{align}
 {\partial_t w}  + c_{fc} {\partial_t p} + \overline{\nabla} \cdot  \boldsymbol Q = \tilde{q}_W - \tilde{q}_L.
\label{eq:mbalc3}
\end{align}
Here, we take $c_{fc} > 0$ as $w\geq 0$. Technically, we do not require strict positivity as the regularity of $p$ on the fracture can be shown to be inherited by the regularity of $p$ in the bulk where we have strict positivity of the coefficient in the time derivative term for the pressure. But we will simplify by assuming strict positivity.  \par
There are two places where $w$ appears in the description of the flow model on the fracture. One is in the fluid storage term, namely, $\partial_t(\rho_f w)$ and the second on the flux term where the permeability is a function of width. For the first term, we consider the width to be an unknown function which is natural as it is jump of the displacement on the fracture surface and hence an unknown. However, we linearize the flow equation in the fracture and take the width appearing in the flux description as the permeability to be known. We thus assume it to be a non-negative function. Since the medium is elastic with finite energy, $w$ in the permeability term is assumed to be bounded and vanishing on the boundary of the fracture. 
To make this distinction clear, we rewrite the above equation in the following form
\begin{align}
 {\partial_t w}  + c_{fc} {\partial_t p} + \overline{\nabla} \cdot  \boldsymbol Q = \tilde{q}_W - \tilde{q}_L.
\label{eq:mbalc3}
\end{align}

%berre2019flow, martin2005modeling, brenner2008mathematical
\begin{remark} \label{rem:model1}
 The fracture flow model presented, here,  namely a mixed dimensional model where Darcy flow on the fracture is described on a surface coupled to the bulk Darcy flow equations is widely used in practice (see e.g., \cite{berre2019flow}). We refer to  \cite{martin2005modeling} where similar models are derived for single phase flow. The model here is a simplified one; the pressure on the fracture is equal to the trace of the bulk pressure.  Such models are obtained when the permeability in the fracture is sufficiently large so that the pressure becomes continuous across the fracture width.  
 
 It is easy to motivate the fracture model by considering first a finite width, equi-dimensional model for the fracture subdomain instead of it being a surface. The appearance of the leak-off term $ \tilde q_L$ on the right-hand side follows from  integrating the  flow equation along the transverse direction in the fracture subdomain. Using the continuity of fluxes at the matrix/fracture interfaces, this yields a surface equation with the jump in the matrix flux term as a source term. The flow model on the fracture surface and the interface conditions can be generalized. For example, instead of the continuity of the pressure at the fracture-matrix interface, we may have a jump. With some extra work, the present work may be extended to take care of this generalization. We refer to \cite{kumar2020formal, list2020rigorous} for more discussions on the fracture flow model in such a setting. 
\end{remark}

\begin{remark} \label{rem:model2}
 The permeability in the fracture ${\frac{w^3}{12}}$ will be taken as a given function of a non-negative function $w = w(x,t)$. Here, we assume that the width, $w$, is a given non-negative function of $x, t$ with $w$ vanishing on the boundaries of the fracture surface. This necessitates the use of weighted spaces $H_w^1(\CS)$ as done above. Note that we make a distinction of the time derivative of $w$ appearing in the storage term appearing in the fracture flow model. Here, $w$ is assumed to be an unknown and is given by the jump in the displacement. This should be considered as a linearization of the non-linear flow equation on the fracture surface where the permeability term is considered to be a given known function whereas the time derivative term, which is linear, is an unknown. It can be interpreted as solving the coupled model in a time discrete manner where we take the permeability term from the previous time step. Well-posedness of the fully non-linear model with permeability as an unknown function remains open. 
\end{remark}

\subsection{Interface, boundary, and initial conditions}

\label{subsec:Interface}
\subsubsection{Fluid pressure and poroelasticity stress}
In the case when contact mechanics including friction is ignored, the balance of the  traction vector yields the interface conditions on each side (or face) of $\mathscr C$:
\begin{equation}
 (\boldsymbol \sigma^{\rm por}(\boldsymbol u,p)) \boldsymbol n^{*} = -p_c \boldsymbol n^{*}, \quad * = +, -.
\label{eq:tildsig}
\end{equation}
In view of the continuity of pressure across the fracture, this leads to ($* = +, -$. )
\begin{align}
 [(\boldsymbol \sigma^{\rm por}(\boldsymbol u,p))]_{\CS} \boldsymbol n^{*} = 0, \quad  \quad (\boldsymbol \sigma^{\rm por}(\boldsymbol u,p)) \boldsymbol n^{*} \cdot \boldsymbol n^* = -p_c, \quad (\boldsymbol \sigma^{\rm por}(\boldsymbol u,p)) \boldsymbol n^{*} \cdot \boldsymbol \tau^* = 0. 
\label{eq:tildsig2}
\end{align}
% with the normal component of $\boldsymbol \sigma^{\rm por}$,  
% $ \sigma^{\rm por}_n = \boldsymbol \sigma^{\rm por} \boldsymbol{n} \boldsymbol n$, being a given function of $p_c, w$. 

\subsubsection{Contact mechanics including friction}

The fracture flow model contains the width of the fracture both in the time derivative term and in the permeability description. However, the model equations have no means of ensuring that the width remains nonnegative. This means that the fracture surfaces may penetrate into each other. Moreover, there is no mechanism to stop this penetration. This is clearly unphysical. Motivated by the arguments from tribology, we will append the existing model by using a simple phenomenological law for the fracture surface that employs the normal stiffness of the interface and ensures the normal compliance of the interface. That is, on the fracture surface $\CS$, instead of \eqref{eq:tildsig2}, we instead impose,

  \begin{align}
\label{eq:normalcontact1} \left(\boldsymbol \sigma^{\rm{por}} \boldsymbol n^* \right) \cdot \boldsymbol n^{*} = -p_c - c_n([-[\boldsymbol u]_\CS \cdot \boldsymbol n^{+} - g_0]_{+}^{m_n}).      
  \end{align}  
  
Here, $g_0$ is the width of the fracture in the undeformed configuration and $[\cdot]_{+}$ is the positive cut of the function; $c_n$ is a non-negative constant, $1\leq m_n \leq 3$ for $d=3$ and $m_n < \infty$ for $d = 2$. 
The above modification in the interface condition imposes a resistance force when the fracture surfaces mutually penetrate. The $c_n, m_n$ are constants that characterize the material interface parameters. The restriction on $m_n$ is motivated by the fact that for $m_n+1$, $H^{1/2}(\CS)$ is continuously embedded in $L^{m_n+1}(\CS)$. This will be useful in obtaining the compactness results later. The above model is adapted from the classical work of Martins \& Oden \cite{martins1987existence} where a similar model is developed in the context of contact problems in mechanics. Here the modification includes the effect of flow pressure.

Next we consider the case when friction is also included in the contact mechanics. As mentioned above, we regularize the model by including viscous damping and consider \eqref{eq:bulk_mechanics}. The physical effect that is modelled here includes the effect that during contact, the modulus of the friction stress is bounded by a value that depends on a polynomial power of the penetration. If a limiting value is not attained, no relative sliding occurs, otherwise sliding occurs with a relative velocity opposite to the frictional stress. This is frequently referred as the stick-slip behavior. The following conditions are used to ensure this.

\medskip\medskip

\noindent
\text{Following \cite{martins1987existence}, the tangential component of the stress tensor at the fracture surface satisfies:}
  \begin{align}\label{eq:tangentialfrictioneqn1}
  \text{ If } -[\uu]_\CS \cdot \boldsymbol n^{+} \leq g_0 \text{ then } 
  |\sigma_T(\uu)| = 0, \text{ and }
  \text{ if } -[\uu]_\CS \cdot \boldsymbol n^{+}  > g_0 \text{ then } 
   |\sigma_T(\uu)| \leq c_T[(-[\uu]_\CS \cdot \boldsymbol n^{+} - g_0)_{+}]^{m_T}.
   \end{align}
{\flushleft 
\text{Moreover, it holds that} 
   \begin{align} \label{eq:tangentialfrictioneqn2}
    \text{ if } -[\uu]_\CS \cdot \boldsymbol n^{+} > g_0  \text{ then } 
    \left \{
    \begin{array}{llll}
     |\sigma_T(\uu)| < c_T[(-[\uu]_\CS \cdot \boldsymbol n^{+} - g_0)_{+}]^{m_T}  \Rightarrow \partial_t{\boldsymbol u}_T   = 0 & \\
      |\sigma_T(\uu)| = c_T[(-[\uu]_\CS \cdot \boldsymbol n^{+} - g_0)_{+}]^{m_T} \text{ and there exists } \lambda \geq 0, \partial_t{\uu}_T   = -\lambda \sigma_T(\uu).& 
  \end{array} \right.
\end{align}}
    
\medskip

\noindent    
In the reference \cite{martins1987existence}, the authors have considered a contact problem with a rigid foundation. This is different from our setting where we have a fractured medium and hence we have elasticity equation on both sides of the contact surface. Therefore, we have adapted these standard conditions of contact mechanics to our present case. Accordingly, instead of the normal component of the displacement, we have the jump of the normal component of the displacement at the fracture surface appearing in the above conditions. See also \cite{sofonea2012mathematical} for further discussion on friction and normal compliance models describing  contact mechanics in a rigid foundation case.  \par 
Here, we impose the same restriction on $m_T$ as on $m_n$, that is, $1\leq m_T \leq 3$ for 3D, and $m_T < \infty$ in 2D.  The first part in \eqref{eq:tangentialfrictioneqn1} describes the effect that there is no tangential stress when the fracture surfaces are either not in contact or merely touching without any effective compressive  stress. The second part provides a bound on the tangential stress when there is a compressive force and the width of the fracture goes below zero. The next statement in \eqref{eq:tangentialfrictioneqn2}  corresponds to slip and is a generalization of Coulomb's law in case slip takes place. The first statement suggests that there is zero slip if the tangential stress is less than a certain bound dependent on the displacement jump. It  models the onset of slip  that takes place when the tangential stress reaches a certain threshold value given by a certain polynomial function of displacement jump. At this threshold value,  the slip takes place in the direction of tangential stress and is a linear function of slip. 
To see that it is indeed a generalization of Coulomb's law, consider the case of no fluid being present, that is, $p_c =0$. In this case 
\[\mu = C|\sigma_n|^{\alpha_c}, \quad \alpha_c = \dfrac{m_T}{m_n} - 1\quad \text{ and } \quad C = \dfrac{c_T}{c_n^{\frac {m_T} {m_n}}} \] 
from where for $m_T = m_n$ we get $\alpha_c =0$ recovering Coulomb's law with $\mu$ being constant. \\
We will need the following quantity (Section \ref{sec:contact2}) later. Define,
\[j(\uu, \vv) = \int_{\CS} c_T[(-[\uu]_{\CS} \cdot \boldsymbol n^{+} - g_0)]^{m_T}|\vv_T|ds, \uu, \vv \in \V. \]
We note that $j$ is convex with respect to the second argument. \par
%We also let $\partial_2 j(\uu, \dot \uu) \in \V \times \V$ denote the partial subdifferential of $j$ with respect to the second argument. 
\par The reduced conservation of mass at the interface gives 
\begin{equation}
\frac{1}{\mu_f}[\boldsymbol K(\nabla\,p - \rho_{f,r} g \nabla\,\eta)]_{\mathscr C}\cdot \boldsymbol n^+ = \tilde q_L.
\label{eq:jumpp}
\end{equation}
The above interface condition results from our assumption of fracture as a surface. The net flux from the matrix to the fracture side appears as a right hand side in the fracture flow model and is given as a leakage term. We refer to the Remark \ref{rem:model1} for motivating the interface condition above. 

We need boundary conditions for the pressure and the displacement. For both, pressure and displacement, we take homogeneous Dirichlet boundary condition. This is taken for simplicity. It suffices to have a strictly positive measurable set on the boundary having Dirichlet boundary condition and for the remaining, we can prescribe a given stress, that is, Neumann boundary conditions. This is true for both the pressure $p$ as well as the displacement $\uu$. 

For the initial conditions, we prescribe the pressure $p_0$ at $t=0$. We take $p_0 \in H^1(\Omega)$. Moreover, initial displacement $\uu(t =0) = \uu_0, \partial_t \uu(t =0) = \uu_1$ is obtained by solving the elasticty equation \eqref{eq:bulk_mechanics}. Elliptic regularity thus implies $\uu_0 \in H^1(\Omega)^d$. The initial value of $w$ is obtained by computing the jump of displacement of $\uu$.
$$w(t =0) = -[\uu_0]_\mathscr C \cdot \boldsymbol n^+$$
 Moreover, for compatibility, we require that  the following initial condition is satisfied
\[\Big(\Big(c_f \phi_0 + \frac{1}{M}\Big) p + \alpha \nabla \cdot \uu\Big)\Big|_{t =0} = \Big(c_f + \frac{1}{M}\Big)\partial_t p_0 + \alpha \nabla \cdot \uu_0. \]
\subsection{Closed system of equations}

We arrive at the complete problem statement, called Problem (Q): Find pressure $p$, displacement $\uu$ satisfying the following set of equations
\begin{eqnarray*}
\partial_{tt} {\boldsymbol u} - \gamma \nabla \cdot \Big (\boldsymbol \varepsilon(\partial_t \uu) + (\nabla \cdot \partial_t \uu) \boldsymbol I \Big ) -  \nabla \cdot \, \Big ( 2G \boldsymbol{\varepsilon(\uu)} + \lambda (\nabla \cdot \uu) \boldsymbol I - \alpha  p  \Big) &=& \boldsymbol f \,\ \mbox{in}\ \Omega \setminus \mathscr C,\\[0.25cm]
% \boldsymbol \sigma^{\rm por}(\boldsymbol u,p) &=& 2G \boldsymbol{\varepsilon(\uu)} + \lambda (\nabla \cdot \uu) \boldsymbol I - \alpha\,p\,\boldsymbol I \,\ \mbox{in}\ \Omega \setminus \mathscr C, \\
  \left(\varphi_0 c_f + \frac{1}{M}\right)  {\partial_t p}  + \nabla\cdot \Big (-\rho_f  \frac{\boldsymbol K}{\mu_f}\big(\nabla\,p - \rho_f g \nabla\,\eta\big) \Big) &=& q \,\ \mbox{in}\ \Omega \setminus \mathscr C,\\
(Q) \hspace{1.6cm}  \partial_t(-\rho_f [\boldsymbol u]_\CS \cdot \boldsymbol n^{+}) - \overline{\nabla} \cdot \Big( \frac{ w^3}{12 \mu_f}(\overline{\nabla}\,p - \rho_{f} g \overline{\nabla}\,\eta)\Big) &=& \tilde q_W - \tilde q_L \,\ \mbox{in}\  \mathscr C,\\
 \hspace{1.6in} 
    % \text{ if } [\uu]_\CS  \leq g  \text{ then }  
    %  |\sigma_T(\uu)| = 0; 
    % \text{ if } [\uu]_\CS > g  \text{ then }  
    %  |\sigma_T(\uu)| = c_T[([\uu]_\CS - g)_{+}]^{m_T}  \text{ and there exists } \lambda \geq 0, \dot{\uu}_T  = -\lambda \sigma_T(\uu) \\
 %\left(\boldsymbol \sigma^{\rm{por}} \boldsymbol n^* \right) \cdot \boldsymbol n^{*} = -p_c - c_n[([\boldsymbol u]_\CS - g]_{+}^{m_n}.
% \boldsymbol \sigma^{\rm por}_\tau =  \boldsymbol \sigma^{\rm por}_\tau ([\boldsymbol{\dot{u}}]_{\mathscr{C}}, \theta) &=& \sigma^{\rm por}_n \mu([\boldsymbol{\dot{u}}]_{\mathscr{C}}, \theta) \boldsymbol{v}.
%  \\
\frac{1}{\mu_f}[\boldsymbol K(\nabla\,p - \rho_{f,r} g \nabla\,\eta)]_{\mathscr C}\cdot \boldsymbol n^+ &=& \tilde q_L \,\ \mbox{on}\ \mathscr C,\\
%w &=& -[\boldsymbol u]_\mathscr C \cdot \boldsymbol n^+.
\end{eqnarray*}
The normal compliance in the normal contact as given  in \eqref{eq:normalcontact1} on the fracture matrix interface $\CS$ can be written as, for $\,\ * = +,-$:
  \begin{align}
    \label{eq:normalcontact2} \left(  \Big (\gamma \boldsymbol \varepsilon(\partial_t \uu) + \gamma \nabla \cdot (\partial_t \uu) \boldsymbol I  + 2G \boldsymbol{\varepsilon(\uu)} + \lambda (\nabla \cdot \uu) \boldsymbol I - \alpha  p \boldsymbol I \Big) \boldsymbol n^* \right) \cdot \boldsymbol n^{*} = -p_c - c_n\Big(-[\boldsymbol u]_\CS \cdot \boldsymbol n^{+} - g_0\Big)_{+}^{m_n}.  
  \end{align}
    
Moreover, the tangential component of the stress tensor at the fracture surface satisfies \eqref{eq:tangentialfrictioneqn1} - \eqref{eq:tangentialfrictioneqn2}. 

% $\text{If } [\uu]_\CS  \leq g_0  \text{ then }  |\sigma_T(\uu)| = 0. $ Moreover,  \\[0.7em]

% ${\text{if } [\uu]_\CS > g_0} \text{ then } 
%   \left\{
%   \begin{array}{ll}
%      |\sigma_T(\uu)| \leq c_T[([\uu]_\CS - g_0)_{+}]^{m_T} \text{  and  } & \\[0.5em]
%      |\sigma_T(\uu)| < c_T[([\uu]_\CS - g_0)_{+}]^{m_T} \text{ then } \dot{\boldsymbol u}_T = 0 & \\[0.5em]
%      |\sigma_T(\uu)| = c_T[([\uu]_\CS - g_0)_{+}]^{m_T}  \text{ and there exists } \lambda \geq 0, \partial_t \uu_T = -\lambda \sigma_T(\uu).
%   \end{array} \right.
% $ \\[1em]

\noindent Here, $w(x,t)$ is a given function that vanishes on $\partial \CS$ and stays positive otherwise. The above system is complemented with suitable initial and boundary conditions as prescribed in Section \ref{subsec:Interface}. It is the Biot--Allard system in the matrix with certain simplifications and coupled to the fracture flow model. The first equation is the dynamical elasticity equation that includes the inertia term. The presence of $p$ in the elasticity equation takes into account the pore pressure of the fluid. The second  equation is the flow equation for the porous matrix. Here, $\varphi^*$ contains the term $\nabla \cdot {\boldsymbol u}$ and measures the change in porosity due to mechanical deformations. The next equation is the flow equation on the fracture. The right-hand side term $\tilde{q}_L$ accounts for the leakage. The remaining equations are interface conditions for the pressure and the stresses { on the fracture's faces}.

\begin{remark}\label{rem:ye_paper}
We contrast present work with the setting that is commonly used as in \cite{Brazda2017, deHoop2020} to account for a compressively pre-stressed case in the absence of fluid.  Following Brazda \etal \cite{Brazda2017}  in a prestressed Earth with pre-stress $\tsT^0$   
the gravitational potential $\phi^0$ satisfies Poisson's equation
\begin{equation}
    \Delta\phi^0=4\pi \Grav\rho^0,
    \label{eq:grav pot}
\end{equation}
with $\rho^0$ the initial density distribution of Earth,
and $\Grav$ Newton's universal constant of gravitation.
The equilibrium condition for the initial steady state is 
\begin{equation}
    \rho^0\nabla\phi^0=\nabla\cdot\tsT^0.
    \label{eq:ini stress}
\end{equation}
In this case, it is reasonable to assume that the width of the fracture is zero. The force equilibrium conditions are then described by  (\eg \cite[(4.57)]{Brazda2017}).
\begin{equation}
    \left\{
        \arraycolsep=1.4pt\def\arraystretch{1.7}
        \begin{array}{rl}
            \jmp{\ven\cdot\veu} = & 0 ,
            \\
            \jmp{\vetau_1(\veu)+\vetau_2(\veu)} = & 0 ,
            \\
            \vetau_\frc-\big(\ven\cdot \tsT^0 
            +\vetau_1(\veu)+\vetau_2(\veu)\big)_\TT = & 0 ,
        \end{array}
        \right.
        \quad \mbox{on } \mathcal{C},
    \label{eq:fric bc}
\end{equation}
with
\begin{equation}
    \left\{
        \arraycolsep=1.4pt\def\arraystretch{1.7}
        \begin{array}{rl}
            \vetau_1(\veu)=&\ven\cdot(\tsLaINI:\nabla\veu) ,
            \\
            \vetau_2(\veu)=&
            -\overline{\nabla} \cdot (\veu(\ven\cdot\tsT^0)) ,
        \end{array}
        \right.
    \label{eq:fric bc var}
\end{equation}
both of which are linear functions depending on $\veu$,
and the surface divergence is defined by $\overline{\nabla}\cdot\vef
=\nabla\cdot\vef-(\nabla\vef\cdot\ven)\cdot\ven$.  
{
Let $\sigma$ denote the magnitude of normal stress with positive sign
denoting compressive force:
}
%Indeed, $\sigma$ is linearly depending on $\veu$ by 
\begin{equation}
    \sigma(\veu)= -\ven\cdot \big(\ven\cdot \tsT^0
    +\vetau_1(\veu)+\vetau_2(\veu)\big).
    \label{eq:def sigma}
\end{equation}

The prestressed elasticity tensor is a linear map  $\tsLaINI: \mathbb{R}^{3 \times 3} \mapsto \mathbb{R}^{3 \times 3}$ such that $\tsLaINI:\nabla\veu$ represents the first Piola-Kirchhoff stress perturbation. In accordance with the Amontons-Coulomb law (\cf \cite[eq.\ (4a)]{Ruina1983}), the relation between the friction force $\tau_\frc$ and the compressive normal stress $\sigma$ satisfies
\begin{equation}
    \tau_\frc = \sigma \,\mu_f(\partial_t {\boldsymbol u}_T).
    \label{eq:frictionlaw}
\end{equation}
In the above, $\mu_f(\partial_t {\boldsymbol u}_T)$ is 
friction coefficient and is taken to be positive and bounded.   There are several  differences between this setting and our setting as evident from \eqref{eq:fric bc}. To compare with the above setting, we put $\tsT^0 \equiv 0$. The first contact condition in \eqref{eq:fric bc} yield zero width for the fracture whereas in our case, we take the width, computed as the normal jump of the displacement at the fracture surface, to be an unknown. In the above model, the two surfaces of the fracture are always in contact with each other as justified due to the presence of compressive pre-stress. In the current setting, the width may be non-zero. The \eqref{eq:fric bc}$_2$ refers to the continuity of forces across the fracture surface. In our case, this is modified to include the fluid pressure which acts even when the width is zero. This is an important difference with the above setting as we include the fracture fluid pressure that impacts the net normal stress in the poro-elastic case. Moreover, the friction as given in \eqref{eq:frictionlaw} is kinetic type as it models the slip motion unlike in the current setting of friction where the static friction acts until a critical value of shear stress is reached beyond which the kinetic friction model becomes active. This models the stick-slip behavior and is described by a variational inequality instead of an equation as in the above case. Moreover, we consider a normal compliance, that is, the normal stress $\sigma$ in our setting is taken as a function of the jump in the displacement $\uu$. This relation helps in getting the required estimates as the function space for $\sigma$ is in $H^{-1/2}(\CS)$ whereas the functional relation of $\sigma$ in terms of $\uu$ exploits that $\uu$ is in  $H^{1/2}(\CS)$ space so that the regularity improves and the integrals involved in the weak form are well-defined on the fracture surface. Finally, in the above setting, we include rate and state law where we introduce a state variable $\psi$ that measures the average contact maturity, 
\begin{equation}
    \tau_\frc=\sigma \,\mu_f(\partial_t {\boldsymbol u}_T,\psi).
    \label{eq:frictionlaw2}
\end{equation}
In the above, $\mu_f(\partial_t{\boldsymbol u}_T,\psi)$ is the 
friction coefficient but now depending on the state variable $\psi$ as well with this evolving with time following an ordinary differential relation.
% \begin{equation}
%     \dot \psi + 
%     \opG\big(\sigma,\ddt\sigma,\dot{\boldsymbol u}_T,\psi\big)=0 .
%     \label{eq:state ODE}
% \end{equation}
 We ignore the rate and state relationship in our work. However, both the presence of pre-stress and the rate and state friction model can be incorporated in our setting and will be considered in the future work. For further discussions on models for friction in the traditional setting of contact mechanics, we refer to the monographs \cite{han2002quasistatic, sofonea2012mathematical}.
% Amontons-Coulomb law states that the magnitude of instantaneous friction force is proportional to the compressive normal stress magnitude in the way that (\cf \cite[eq.\ (4a)]{Ruina1983}).
%     \begin{equation}
%         \opF\big(\sigma,s,\psi\big)=\sigma \,f(s,\psi).
%         \label{eq:Coulomb law}
%     \end{equation}
%     \label{ass:Coulomb law}

\end{remark}

\subsection{Simplified system: Absence of contact mechanics and friction}
Before we treat the full system of equations in Problem (Q), we first consider a simplified system of equations where we ignore the contact mechanics including friction. We call this  Problem (Q$_0$). The analysis of Problem (Q$_0$) is simpler but it contains the essence of the arguments. Furthermore, this is of independent interest itself as it extends the previous work of \cite{girault2015lubrication} to the case of fully dynamic Biot equation where inertial term is also included. The enrichment of the model by including contact mechanics and friction becomes an add-on to the displacement equation and is then treated using the existing theory of contact mechanics, e.g., \cite{martins1987existence} though here includes the coupling with the fluid flow. \\

\noindent {\bf{Problem (Q$_0$):}}\\[0.6em]
Find pressure $p$, displacement $\uu$ satisfying the following set of equations
\begin{eqnarray*}
\partial_{tt} {\boldsymbol u}  -  \nabla \cdot \,\left ( 2G \boldsymbol{\varepsilon(\uu)} + \lambda (\nabla \cdot \uu) \boldsymbol I - \alpha  p \right) &=& \boldsymbol f \,\ \mbox{in}\ \Omega \setminus \mathscr C,\\[0.25cm]
% \boldsymbol \sigma^{\rm por}(\boldsymbol u,p) &=& 2G \boldsymbol{\varepsilon(\uu)} + \lambda (\nabla \cdot \uu) \boldsymbol I - \alpha\,p\,\boldsymbol I \,\ \mbox{in}\ \Omega \setminus \mathscr C, \\
  \left(\varphi_0 c_f + \frac{1}{M}\right)  {\partial_t p}  + \nabla\cdot \left (-\rho_f  \frac{\boldsymbol K}{\mu_f}\big(\nabla\,p - \rho_f g \nabla\,\eta\big) \right) &=& q \,\ \mbox{in}\ \Omega \setminus \mathscr C,\\
(Q_0) \hspace{1.6cm}  \partial_t(-\rho_f [\boldsymbol u]_\CS \cdot \boldsymbol n^{+}) - \overline{\nabla} \cdot \Big( \frac{ w^3}{12 \mu_f}(\overline{\nabla}\,p_c - \rho_{f} g \overline{\nabla}\,\eta)\Big) &=& \tilde q_W - \tilde q_L \,\ \mbox{in}\  \mathscr C,\\
 \hspace{1.6in} 
    % \text{ if } [\uu]_\CS  \leq g  \text{ then }  
    %  |\sigma_T(\uu)| = 0; 
    % \text{ if } [\uu]_\CS > g  \text{ then }  
    %  |\sigma_T(\uu)| = c_T[([\uu]_\CS - g)_{+}]^{m_T}  \text{ and there exists } \lambda \geq 0, \dot{\uu}_T  = -\lambda \sigma_T(\uu) \\
 %\left(\boldsymbol \sigma^{\rm{por}} \boldsymbol n^* \right) \cdot \boldsymbol n^{*} = -p_c - c_n[([\boldsymbol u]_\CS - g]_{+}^{m_n}.
% \boldsymbol \sigma^{\rm por}_\tau =  \boldsymbol \sigma^{\rm por}_\tau ([\boldsymbol{\dot{u}}]_{\mathscr{C}}, \theta) &=& \sigma^{\rm por}_n \mu([\boldsymbol{\dot{u}}]_{\mathscr{C}}, \theta) \boldsymbol{v}.
%  \\
\frac{1}{\mu_f}[\boldsymbol K(\nabla\,p - \rho_{f,r} g \nabla\,\eta)]_{\mathscr C}\cdot \boldsymbol n^+ &=& \tilde q_L \,\ \mbox{on}\ \mathscr C.
%w &=& -[\boldsymbol u]_\mathscr C \cdot \boldsymbol n^+.
\end{eqnarray*}
The normal  contact on the fracture matrix interface $\CS$ ignores the normal compliance term and includes only the fluid pressure:
  \begin{align}
    \label{eq:normalcontact3} \left( (2G \boldsymbol{\varepsilon(\uu)} + \lambda (\nabla \cdot \uu) \boldsymbol I - \alpha  p \boldsymbol I) \boldsymbol n^* \right) \cdot \boldsymbol n^{*} = -p_c,  \quad  \,\ * = +,-.  
  \end{align}
Moreover, the tangential component of the stress tensor on  fracture matrix interface $\CS$ satisfies
$$\sigma_T(\uu) \equiv 0.$$
The interface condition across the fracture interface assumes the continuity of pressure, that is on $\CS$, 
$$p = p_c$$
where the pressure $p$ on $\CS$ is taken in the sense of traces.

% \begin{remark}
% In the above model, the mechanical deformations of both the porous matrix and that of fracture are included. 

% \end{remark}

\section{Weak formulation of  Problem (Q$_0$)} \label{sec:weak}

We introduce a weak formulation of problem (Q$_0$) and will analyze this. Note that this is a linear problem. It is a straightforward exercise to verify that provided the  data (right-hand side, coefficients, initial and boundary conditions) and the solution are sufficiently regular, the strong and weak formulations are equivalent.  Below we present two formulations that will be shown to be equivalent. 

\subsection{Formulation 1}

To simplify the notation, we omit the spaces in the inner products; if the domain of integration is not indicated, then it is understood that the integrals are taken over $\Omega^+ \cup \Omega^-$.  The conformal variational formulation of Problem (Q$_0$) reads: For given ${\boldsymbol f} \in L^2((\Omega \setminus \CS) \times ]0,T[)^d$, $\tilde q \in L^2(\Omega \times ]0,T[)$, and  $\tilde q_L \in L^2(0,T;H^{1}(\CS)')$, find  $\boldsymbol u \in L^\infty(0,T;\V)$, $\partial_t{\boldsymbol u} \in L^\infty(0,T;L^2(\Omega)) $, $\partial_{tt} \boldsymbol u \in L^2(0,T;\V^{'})$, \\ $p \in L^\infty(0,T;L^2(\Omega)) \cap L^2(0,T;Q)$, $p_c \in L^2(0,T;H^{\frac{1}{2}}(\CS))$, such that 
\begin{equation} 
\text{ for all } {\boldsymbol v} \in \V\,,\,(\partial_{tt} {\boldsymbol u}, {\boldsymbol v}) + 2 G \big(\varepsilon({\boldsymbol u}), \varepsilon({\boldsymbol v})\big) + \lambda \big(\nabla \cdot {\boldsymbol u},\nabla\cdot {\boldsymbol v}\big) 
- \alpha \big( p ,\nabla\cdot {\boldsymbol v}\big)
   +\big(p_c ,[{\boldsymbol v}]_{\CS} \cdot {\boldsymbol n}^+\big)_\CS = \big( {\boldsymbol f} ,{\boldsymbol v}\big)\,,\, 
\label{eq:varumix}
\end{equation}
\begin{align}
\label{eq:varpmix1}   \text{for all } \theta \in Q\,,\,\Big({\partial_t } \Big(\Big(\frac{1}{M} + c_f \varphi_0\Big) p +\alpha  \nabla\cdot {\boldsymbol u} \Big), \theta\Big)
&+ \frac{1}{\mu_f}\big( \boldsymbol K \nabla p - \rho_{f,r} g {\nabla} \eta ,\nabla \theta \big) = \big(\tilde q, \theta \big) + \langle \tilde{q}_L,\theta \rangle_{\CS},\\
\label{eq:varpmix2}  \text{for all } \theta_c \in Q\,,\, (c_{fc} \partial_t p_c, \theta_c)_{\CS} + \dfrac{d}{dt}  (-[{\boldsymbol u}]_\CS \cdot {\boldsymbol n}^+, \theta_c )_\CS &
\nonumber\\
+ \frac{1}{12 \mu_f} ( w^{\frac 3 2} (\overline{\nabla} p_c - \rho_{f,r} g \overline{\nabla} \eta)&, w^{\frac 3 2} \overline{\nabla} \theta_c )_\CS =  ( \tilde q_W , \theta_c )_\CS,
- \langle \tilde{q}_L, \theta_c \rangle_\CS ,
\end{align}
 subject to the initial condition,      
\begin{align}
\label{eqn:initial}
\Big(\Big(\frac{1}{M} + c_f \varphi_0\Big)p + \alpha \nabla \cdot {\boldsymbol u}\Big)\Big|_{t=0} = \Big(\frac{1}{M} + c_f \varphi_0\Big)p_0 + \alpha \nabla \cdot {\boldsymbol u}_0.
\end{align}
This weak formulation is obtained starting from the problem (Q$_0$) using the standard method of taking smooth test functions and then performing partial integration and using the boundary conditions. Note that in \eqref{eq:varpmix2}, the test function $\theta_c$ is the trace of $\theta$ on $\CS$ and $\gamma =0$. 
%This derivation is standard and will be left as an exercise.

\subsection{Formulation 2: Reduced formulation}

Here the leakage term is eliminated by adding the two mass balance equations in the fracture and in the matrix. Note that the leakage term $\tilde{q}_L$ has positive and negative signs in the mass balance equation in bulk and the fracture respectively, thereby cancelling while the two are summed up. The regularity of $p$ and $\theta$ allows us to take traces of these functions and the terms on the $\CS$ are thus well defined. The reduced formulation therefore reads:
For given ${\boldsymbol f} \in L^2((\Omega \setminus \CS) \times ]0,T[)^d$, $\tilde q \in L^2(\Omega \times ]0,T[)$, and $\tilde q_l \in L^2(0,T;H^{1}(\CS)')$, find  $\boldsymbol u \in L^\infty(0,T;\V)$, $\partial_t{\boldsymbol u} \in L^\infty(0,T;L^2(\Omega)) $, $\partial_{tt} \boldsymbol u \in L^2(0,T;\V^{'})$,  $p \in L^\infty(0,T;L^2(\Omega)) \cap L^2(0,T;Q)$, with $p_c$ being the trace of $p$, such that
\begin{equation} 
\text{for all } {\boldsymbol v} \in \V\,,\,(\partial_{tt} {\boldsymbol u}, {\boldsymbol v}) + 2 G \big(\varepsilon({\boldsymbol u}), \varepsilon({\boldsymbol v})\big) + \lambda \big(\nabla \cdot {\boldsymbol u},\nabla\cdot {\boldsymbol v}\big) 
- \alpha \big( p ,\nabla\cdot {\boldsymbol v}\big)
   +\big(p_c ,[{\boldsymbol v}]_{\CS} \cdot {\boldsymbol n}^+\big)_\CS  = \big( {\boldsymbol f} ,{\boldsymbol v}\big)\,,\, 
\label{eq:varu2}
\end{equation}
\begin{align}
\text{ for all } \theta \in Q\,,\,\Big({\partial_t }\Big(\Big(\frac{1}{M} + c_f \varphi_0\Big) p + \alpha  \nabla\cdot {\boldsymbol u} \Big) ,\theta \Big)
+ \frac{1}{\mu_f}\big( \boldsymbol K \nabla p  - \rho_{f,r} g {\nabla} \eta,\nabla \theta \big) +(c_{fc} \partial_t p_c, \theta)_\CS
\\
 - \dfrac{d}{dt}  ([{\boldsymbol u}]_\CS \cdot {\boldsymbol n}^+, \theta)_\CS + \frac{1}{12 \mu_f} \big ( w^{\frac 3 2} (\overline{\nabla} p_c - \rho_{f,r} g \overline{\nabla} \eta), w^{\frac 3 2} \overline{\nabla} \theta \big )_\CS = \big(\tilde q, \theta \big).
\label{eq:varp2}
\end{align}

% \subsection{Formulation 3: Further reduced formulation}

% Finally, we consider the displacement equation \eqref{eq:varu2} and solve this with given $p$. This provides us a mapping from $p$ to the solution of the displacement equation $\boldsymbol u$. Precisely, for $p \in L^\infty(0,T;L^2(\Omega)) \cap L^2(0,T;Q)$, we obtain a unique $\uu \in  L^\infty(0,T;\V), \partial_t{\boldsymbol u} \in L^\infty(0,T;L^2(\Omega))$. Moreover, the linearity of elastic equation implies that the solution operator $\mathcal{S}_u: p \rightarrow \uu$ is linear. 

% We accordingly get: For given ${\boldsymbol f} \in L^2((\Omega \setminus \CS) \times ]0,T[)^d$, $\tilde q \in L^2(\Omega \times ]0,T[)$, and $\tilde q_l \in L^2(0,T;H^{1}(\CS)')$, find  $p \in L^\infty(0,T;L^2(\Omega)) \cap L^2(0,T;Q)$, with $p_c$ being the trace of $p$, such that
% \begin{align}
% \text{for all } \theta \in Q\,,\,\Big({\partial_t }\Big(\Big(\frac{1}{M} + c_f \varphi_0\Big) p +\alpha  \nabla\cdot {\boldsymbol u} \Big) ,\theta \Big)
% + \frac{1}{\mu_f}\big( \boldsymbol K \nabla p -  \rho_{f,r} g {\nabla} \eta,\nabla \theta \big) +(c_{fc} \partial_t p_c, \theta)_\CS
% \\
%  - \dfrac{d}{dt} ([{\boldsymbol u}]_\CS \cdot {\boldsymbol n}^+, \theta)_\CS + \frac{1}{12 \mu_f} \big ( w^{\frac 3 2} (\overline{\nabla} p_c - \rho_{f,r} g \overline{\nabla} \eta), w^{\frac 3 2} \overline{\nabla} \theta \big )_\CS = \big(\tilde q, \theta \big),
% \label{eq:varp3}
% \end{align}

\subsection{Equivalence of formulations}

We first prove a preliminary result. We let $M = H_w^1(\CS)^\prime$ and define the bilinear form 
\[
\text{for all } q \in Q, \text{ for all } \xi \in M,\; \; b(q, \xi) = \langle \xi, q_c \rangle_\CS.
\]

\begin{proposition} \label{prop:coercivity}
There exists a constant $C$ such that 
\begin{align*}
\text{for all } \xi \in M, \quad \sup_{q \in Q} \dfrac{b(q,\xi)}{\|q\|_Q} \geq \dfrac{1}{C}\|\xi\|_{H_w^1(\CS)^\prime}.    
\end{align*}
\end{proposition}

We just sketch the proof for the sake of completeness. Let $C$ be the continuity constant of the extension mapping $E: H_w^1(\CS) \mapsto Q$; $C$ depends only on $\Omega$ and $\CS$, and $E$ satisfies 
\begin{align*}
\text{for all } \theta \in H_w^1(\CS), \quad \|E(\theta)\|_Q \leq C\|\theta\|_{H_w^1(\CS)}.    
\end{align*}
It follows that
\begin{align*}
\text{for all }    \xi \in M :\quad \|\xi\|_{H_w^1(\CS)^\prime} = \sup_{z \in H_w^1(\CS)}  \dfrac{\langle \xi, z \rangle_\CS}{\|z\|_{H_w^1(\CS)}} = 
\dfrac{\langle \xi, E(z) \rangle_\CS}{\|z\|_{H_w^1(\CS)}} \leq C \sup_{q \in Q} \dfrac{b(q,\xi)}{\|q\|_Q}.
\end{align*}
This completes the argument.

We state the equivalence of the two formulations, formulation 1 and the reduced formulation 2. 

\begin{lemma}
Let the data $\boldsymbol f \in L^\infty (0,T; L^2(\Omega)^d), \tilde{q} \in L^2(\Omega \times (0,T))$, $\tilde{q}_W \in L^2(\CS \times (0,T))$, $p_0 \in Q$, and $\boldsymbol u_0, \uu_1 \in V$. Let $p \in L^\infty(0,T;L^2(\Omega)) {\displaystyle \cap} L^2(0,T;Q) $, $\boldsymbol u \in L^\infty(0,T;\V)$, $\partial_t{\boldsymbol u} \in L^\infty(0,T;L^2(\Omega)) $, $\partial_{tt} \boldsymbol u \in L^2(0,T;\V^{'})$, $\tilde{q}_L \in W^{-1,\infty}(0,T;H_w^1(\CS)^\prime$ solve formulation 1. Then $p, \uu$  solve formulation 2. Conversely, $p \in L^\infty(0,T;L^2(\Omega)) {\displaystyle \cap} L^2(0,T;Q) $, $\boldsymbol u \in L^\infty(0,T;\V)$, $\partial_t{\boldsymbol u} \in L^\infty(0,T;L^2(\Omega)) $, $\partial_{tt} \boldsymbol u \in L^2(0,T;\V^{'})$, solving formulation 2 also solve formulation 1.
\end{lemma}

\begin{proof}
Take for simplicity $\tilde{q}_W =0$. Using formulation 1, by adding the two equations \eqref{eq:varpmix1} and \eqref{eq:varpmix2} for the pressure equations  in the matrix and on the fracture enables elimination of $\tilde q_L$. We obtain \eqref{eq:varp2}. This completes the derivation of obtaining formulation 2 starting from formulation 1. 
% \begin{align}
% \nonumber\forall \theta \in H^1(\Omega)\,,\,\Big(\frac{\partial }{\partial t}\big((\frac{1}{M} + c_f \varphi_0) p +\alpha  \nabla\cdot {\boldsymbol u} \big) ,\theta \Big)
% + \frac{1}{\mu_f}\big( K \nabla p,\nabla \theta \big) +(c_{fc} \partial_t p_c, \theta)_\CS\\
%  + \dfrac{d}{dt}  ([{\boldsymbol u}]_\CS \cdot {\boldsymbol n}^+, \theta)_\CS + \frac{1}{12 \mu_f} \big ( w^{\frac{3}{2}} (\overline{\nabla} p_c - \rho_{f,r} g \eta), w^{\frac{3}{2}} \overline{\nabla} \theta \big )_\CS = \big(\tilde q, \theta \big).
% \label{eq:varp4}
% \end{align}

Next, we show the converse. We begin with formulation 2. Let us define the kernel of $b(\cdot, \cdot) \in Q$ as
\[
Q_0 = \{q \in Q; \text{ for all } \xi \in M, b(q, \xi) = 0 \} = \{q \in Q; q_c = 0\}.
\]
% We define the polar set  $Q_0^0$ of $Q_0$ by
% \[ Q_0^0 = \{\xi \text{ for all } \theta \in Q_0, \langle \xi, \theta \rangle = 0\}. \]
We observe that by integrating \eqref{eq:varp2} over $(0,t)$, testing with $\theta \in Q_0$, we have 
\begin{align*}
\int_{\Omega^+ \cup\, \Omega^-} \Big(\Big(\frac{1}{M} + c_f \varphi_0\Big)p+ \alpha \nabla \cdot \boldsymbol u \Big)(t) \theta dx - \int_{\Omega^+ \cup \Omega^-} \Big(\Big(\frac{1}{M} + c_f \varphi_0\Big)p_0+ \alpha \nabla \cdot \boldsymbol u_0 \Big) \theta dx+\\
+ \int_0^t \int_{\Omega^+ \displaystyle \cup \Omega^-} \Big(\frac{1}{\mu_f}  \boldsymbol K (\nabla p(\tau) - \rho_{f,r} g \nabla \eta ) \cdot \nabla \theta - \tilde{q}(s) \theta \Big) dx d\tau  = 0. 
\end{align*}
{We identify the left-hand side with a linear functional, $l(t) : Q \to \mathbb{R}$, $\ \theta \to \langle l(t), \theta \rangle$} a.e. on $(0,T)$. The equation above therefore can be written as 
\[ \text{ for all } \theta \in Q_0, \langle l(t), \theta \rangle = 0. \]

%, by
%\begin{align*}
%\langle{l(t)}, \theta \rangle =   \int_{\Omega^+ \cup \Omega^-} \left((\frac{1}{M} + c_f \phi_0)p+ \alpha \nabla \cdot \boldsymbol u \right)(t) \theta - \int_{\Omega^+ \cup \Omega^-} \left((\frac{1}{M} + c_f \phi_0)p_0+ \alpha \nabla \cdot \boldsymbol u_0 \right) \theta +\\
%\int_0^t \int_{\Omega^+ \cup \Omega^-} \left(\frac{1}{\mu_f}  K (\nabla p(s) - \rho_{f,r} g \nabla \eta ) \cdot \nabla \theta - \tilde{q}(s) \theta \right).  
%\end{align*}
Denoting by $Q^\prime$ the dual space of $Q$, the above assumptions on $p$ and $\uu$ imply that $l \in L^\infty(0,T;Q^\prime)$. Using Proposition \ref{prop:coercivity}, properties of $H_w^1(\Omega)$ from Theorem \ref{th:weightedspace}, it follows then that there exists a unique $\tilde{Q}_L \in L^\infty(0,T; H_w^1(\CS)^\prime)$ such that a.e. in $(0,t)$, $\langle l, \theta \rangle = b(\theta, \tilde{Q}_L) = \langle \tilde Q_L, \theta_c \rangle_\CS$. 
That is, for all $\theta \in Q$ 
\begin{multline*}
\int_{\Omega^+ \cup\, \Omega^-} \Big(\Big(\frac{1}{M} + c_f \varphi_0\Big)p+ \alpha \nabla \cdot \boldsymbol u \Big)(t) \theta dx - \int_{\Omega^+ \cup \Omega^-} \Big(\Big(\frac{1}{M} + c_f \varphi_0\Big)p_0+ \alpha \nabla \cdot \boldsymbol u_0 \Big) \theta dx +\\
+ \int_0^t \int_{\Omega^+ \displaystyle \cup \Omega^-} \Big(\frac{1}{\mu_f}  \boldsymbol K (\nabla p(\tau) - \rho_{f,r} g \nabla \eta ) \cdot \nabla \theta - \tilde{q}(\tau) \theta \Big) dx d\tau = \langle \tilde{Q}_L, \theta_\CS \rangle_{\CS}. 
\end{multline*}
After differentiating with respect to time,  $\text{for all } \theta \in Q$, we get 
\begin{multline*}
\int_{\Omega^+ \cup \Omega^-} \partial_t\Big(\Big(\frac{1}{M} + c_f \varphi_0 \Big)p+ \alpha \nabla \cdot \boldsymbol u \Big) \theta dx 
\\
+ \int_{\Omega^+ \cup \Omega^-} \Big(\frac{1}{\mu_f}  \boldsymbol K (\nabla p - \rho_{f,r} g \nabla \eta ) \cdot \nabla \theta - \tilde{q} \theta \Big) dx = \langle \partial_t \tilde {Q}_L, \theta_c \rangle_\CS. 
\end{multline*}
By setting $\tilde q_L = \partial_t \tilde{Q}_L$ and comparing this with
\begin{multline}
\text{ for all } \theta \in Q\,,\,\Big({\partial_t}\Big(\Big(\frac{1}{M} + c_f \varphi_0\Big) p +\alpha  \nabla\cdot {\boldsymbol u} \Big) ,\theta \Big)
+ \frac{1}{\mu_f}\big( (\boldsymbol K \nabla p  - \rho_{f,r} g \nabla \eta) ,\nabla \theta \big) \\
  =-(c_{fc} \partial_t p_c, \theta)_\CS  +    (\partial_t[{\boldsymbol u}]_\CS \cdot {\boldsymbol n}^+, \theta)_\CS - \frac{1}{12 \mu_f} \big ( w^{\frac{3}{2}} (\overline{\nabla} p_c - \rho_{f,r} g \overline \nabla \eta), w^{\frac{3}{2}} \overline{\nabla} \theta \big )_\CS + \big(\tilde q, \theta \big), 
  \label{eq:varp5}
\end{multline}
we get 
\begin{align}
-(c_{fc} \partial_t p_c, \theta)_\CS  +   (\partial_t[{\boldsymbol u}]_\CS \cdot {\boldsymbol n}^+, \theta)_\CS - \frac{1}{12 \mu_f} \big ( w^{\frac{3}{2}} (\overline{\nabla} p_c - \rho_{f,r} g  \overline{\nabla}\eta), w^{\frac{3}{2}} \overline{\nabla} \theta \big )_\CS  = \langle \tilde{q}_L, \theta \rangle_\CS.    
\end{align}
\end{proof}
%The equivalence of formulation 2 and formulation 3 is straightforward and is left as an exercise. 
\section{Well-posedness of the Problem ($Q_0$)} \label{sec:wellposedness}

The existence and uniqueness of solution follows the familiar steps. We first establish a finite dimensional approximation of the solution in the conformal spaces. Then we use these approximations as test functions to obtain estimates that are independent of the dimension of the approximation spaces. We note that the estimates are obtained in Hilbert spaces (reflexive Banach spaces) and therefore bounded subsets are weakly sequentially precompact. Accordingly, we can extract a subsequence that is weakly converging in these spaces. The limit functions  satisfy the continuous equations due to the linearity of the equation. 

\subsection{Galerkin procedure for finite dimensional approximation}
We take the formulation 2 so that the leakage term is eliminated. We consider a semi-discrete Galerkin scheme. Let $\{\theta_i\}_{i = 1}^{\infty}$ form a basis for $Q$, that is, for any $m\geq 1$, we approximate $Q$ by $Q_m$ having the basis $\{\theta_i, i\leq m\}$ and $Q = \overline{\bigcup_{m \geq 1} Q_m}$.
We define the Galerkin approximation  $p_m$  of $p$
\[p_m(t) = \sum_{j = 1}^m P_j(t) \theta_j,\]
with $p_j = H^2(0,T), 1\leq j\leq m$. 
Similarly, considering $\{\boldsymbol{w}_i\}_{i = 1}^\infty$ be a sequence of functions $\boldsymbol w_i \in V$ for all $i$ and let  $\boldsymbol w_i, i = 1, \ldots, m$ span the subspace $V_m$ of $V$ with $V = \overline{\bigcup_{m \geq 1} V_m}$. To approximate the solution $\uu_m: \Omega^+ \bigcup \Omega^- \times [0,T] \rightarrow \mathbb{R}^d, d = 2,3$ of mechanics equation \eqref{eq:varumixfr1}, we define 
\[\uu_m(\boldsymbol x,t) :=  \sum_{j=1}^m  U_j(t) \boldsymbol w_j(\boldsymbol x). \]
The solution of the fracture pressure $p_{cm}$ is obtained by taking the trace of $p_m$ on $\CS$.
% For the leakage term, we have
% \[\widetilde{q}_{m,L} (t) = \sum_{j=1}^m Q_j(t)\theta_{j,c}.\]
The coefficients of the approximations are chosen so that the model equations \eqref{eq:varumixfr1}-\eqref{eq:varpmixfr3} lead to the following system of coupled ODEs:
\begin{multline}
\noindent \text{For } \boldsymbol{w}_i, i = 1, \ldots,m,\\
 (\partial_{tt} {{\boldsymbol u}_m}, {\boldsymbol w}_i) + 2 G \big(\varepsilon({\boldsymbol u}_m), \varepsilon({\boldsymbol w_i})\big) + \lambda \big(\nabla \cdot {\boldsymbol u}_m,\nabla\cdot {\boldsymbol w_i}\big)  - \alpha \big( p_m ,\nabla\cdot {\boldsymbol w_i}\big) \\
 +\big(p_{cm} ,[{\boldsymbol w_i}]_{\CS} \cdot {\boldsymbol n}^+\big)_\CS = \big( {\boldsymbol f} ,{\boldsymbol w_i}\big) 
\label{eq:varumix10}    
\end{multline} 
for $\theta_i, i = 1,\ldots,m$, \\[0.4em]
\begin{align}
   \Big({\partial_t}\Big(\Big(\frac{1}{M} + c_f \varphi_0\Big) p_m +\alpha  \nabla\cdot {\boldsymbol u}_m \Big), \theta_i \Big)
+ \frac{1}{\mu_f}\big( \boldsymbol K \nabla p_m - \rho_{f,r} g  \nabla  \eta,\nabla \theta_i \big) &= \big(\tilde q, \theta_i \big) + \langle \tilde{q}_L,\theta_i \rangle_{\CS}, \label{eq:varpmix10} 
\end{align}
and for $\theta_{i,c}, i=1, \ldots, m$,\\  
\begin{multline}
(c_{fc} \partial_t p_{cm}, \theta_{i,c})_{\CS} +   (-\partial_t[{\boldsymbol u}_m]_\CS \cdot {\boldsymbol n}^+, \theta_{i,c} )_\CS + \\
+ \frac{1}{12 \mu_f} ( w^{3/2} (\overline{\nabla} p_{cm} - \rho_{f,r} g \overline \nabla  \eta), w^{3/2} \overline{\nabla} \theta_{i,c} )_\CS =  ( \tilde q_W , \theta_{i,c} )_\CS
- \langle \tilde{q}_L, \theta_{i,c} \rangle_\CS.
\label{eq:varpmix11}
\end{multline}
The initial conditions satisfy
\begin{align*}
    \uu_m(0) &= \uu_0^m = \sum_{i = 1}^m \alpha_i^m \boldsymbol{w}_i \rightarrow \uu_0 \text{ in } V \text{ as } m \rightarrow \infty,\\
    \partial_t \uu_m(0) &= \uu_1^m = \sum_{i = 1}^m \beta_i^m \boldsymbol{w}_i \rightarrow \uu_1 \text{ in } L^2(\Omega) \text{ as } m \rightarrow \infty,
\end{align*}
and analogously for the pressure $p_m(0)$.  This forms a system of ODEs and we can show the existence of a solution  by showing the existence for this system. \\
% For $1\leq i \leq m, \forall \theta_i \in H^1(\Omega)\,,\,$
% \begin{multline}
%  \Big(\big((\frac{1}{M} + c_f \varphi_0) \partial_t p_m +\alpha  \nabla\cdot {\partial_t \boldsymbol u_m} \big) ,\theta_i \Big)
%  + \frac{1}{\mu_f}\big( \boldsymbol K \nabla \partial_t p_m,\nabla \theta_i \big) +(c_{fc}  \partial_t p_{cm}, \theta_i)_\CS \\
%   + ([{\partial_t \boldsymbol u}]_\CS \cdot {\boldsymbol n}^+, \theta_i)_\CS + \frac{1}{12 \mu_f} \big ( w^{\frac{3}{2}} (\overline{\nabla} p_{cm} - \rho_{f,r} g \eta), w^{\frac{3}{2}} \overline{\nabla} \theta_i \big )_\CS 
%  = \big(\tilde q, \theta \big),
% \label{eq:varp4}
% \end{multline}
We define the following matrices

\begin{tabular}{ l l l}
     \hspace{-1cm} $A^{up}_{ij} =  \int_{\Omega^{+} \cup \Omega^{-}} \theta_j \nabla \cdot \boldsymbol w_i dx$, & $C_{\uu ij} = \int_\Omega  \boldsymbol w_i \boldsymbol w_j dx$,  & \hspace{-2cm} $B_{ij} = \int_{\Omega^{+} \cup \Omega^{-}} \alpha \left(\nabla \cdot {\boldsymbol w}_j \right) \left (\nabla \cdot {\boldsymbol w}_i \right) dx$\\[1em]
      \hspace{-1cm}  $C_{p ij} = \int_\Omega \Big(\frac{1}{M}+c_f \varphi_0 \Big) \theta_i \theta_j dx$, &  $C_{pc ij} = \int_\CS c_{fc}   \theta_{i,c} \theta_{j,c} ds$, & \hspace{-2cm} $L_{ij} = \int_{\Omega^{+} \cup \Omega^{-}} \frac{1}{\mu_f} {\boldsymbol K}\nabla \theta_j \cdot \nabla \theta_i dx$, \\[1em]
      \hspace{-1cm}  $M_{ij}^{upc} = -\int_{\CS}[\boldsymbol w_j]_{|\CS} \cdot {\boldsymbol n}^{+} \theta_{i,c} ds$, & $W_{ij} = \int_{\CS} \dfrac{w^3}{12 \mu_f} \bar \nabla \theta_{j,c} \cdot  \bar \nabla \theta_{i,c} ds$, &  {} \\[1em]
      \hspace{-1cm}  $E_{1ij} = \int_{\Omega^{+} \cup \Omega^{-}} 2G {\boldsymbol \varepsilon}({\boldsymbol w}_j) : {\boldsymbol \varepsilon}({\boldsymbol w}_i) dx$, & $E_{2ij} = \int_{\Omega^{+} \cup \Omega^{-}} \lambda ({\nabla \cdot} {\boldsymbol w}_j) ({\nabla \cdot} {\boldsymbol w}_i)dx$. & {} \\[1em] 
\end{tabular}

% \begin{array}{lll}
% $
%  A^{up}_{ij} =  \int_\Omega \theta_j \nabla \cdot \boldsymbol w_i dx, \quad
% A^{upc}_{ij} =  \int_\Omega \theta_j \nabla \cdot \boldsymbol w_i dx, \quad
% B_{ij} = \int_{\Omega^{+} \cup \Omega^{-}} \alpha \left(\nabla \cdot {\boldsymbol w}_j \right) \left (\nabla \cdot {\boldsymbol w}_i \right) dx\\
%  C_{p ij} = \int_\Omega  \theta_i \theta_j dx,  \quad 
% C_{\uu ij} = \int_\Omega  \boldsymbol w_i \boldsymbol w_j dx, \quad L_{ij} = \int_\Omega \frac{1}{\mu_f} {\boldsymbol K}\nabla \theta_j \cdot \nabla \theta_i dx, \\
% M_{ij} = -\int_{\CS}[\boldsymbol w_j]_{|\CS} \cdot {\boldsymbol n}^{+} \theta_{i,c} ds,  \quad 
% W_{ij} = \int_{\CS} \dfrac{w^3}{12 \mu_f} \bar \nabla \theta_{j,c} \cdot  \bar \nabla \theta_{i,c} ds. \\
%  E_{ij} = \int_{\Omega^{+} \cup \Omega^{-}} 2G {\boldsymbol \varepsilon}({\boldsymbol w}_j) : {\boldsymbol \varepsilon}({\boldsymbol w}_i) dx + \int_{\Omega^{+} \cup \Omega^{-}} \lambda ({\nabla \cdot} {\boldsymbol w}_j) ({\nabla \cdot} {\boldsymbol w}_i)dx. 
% $
% \end{array}
We denote $U= [\boldsymbol U_1, \boldsymbol U_2, \ldots, \boldsymbol U_m]^T$; then the elasticity equation results in the equation
\begin{align}
    C_\uu \dfrac{d^2 {\boldsymbol U}}{dt^2} + (E_1 + E_2)  {\boldsymbol U}  + A^{up} P = F.
\end{align}
We denote $P= [p_1, p_2, \ldots, p_m]^T$; then the mass balance equations in fracture and matrix result in the equation
\begin{align}
 (C_p + C_{pc}) \dfrac{d P}{dt} + A^{up} \dfrac{d U}{dt}  + (L+ W)P = Q.    
\end{align}
Define: $\boldsymbol \chi = \dfrac{d \boldsymbol U}{d t}$ to rewrite the above system
\begin{align*}
    C_\uu \dfrac{d {\boldsymbol \chi}}{dt} + (E_1 + E_2)  {\boldsymbol U}  + A^{up} P = F, \\
 (C_p + C_{pc}) \dfrac{d P}{dt} + A^{up} \dfrac{d U}{dt}  + (L+ W)P = Q,\\
  \dfrac{d \boldsymbol U}{d t} - \boldsymbol \chi = 0.
\end{align*}
This is in the form of a linear DAE: 
\begin{align}
    \label{eq:dae1} \mathcal{M} \dfrac{d \boldsymbol \Xi}{dt} + \mathcal{N} \boldsymbol \Xi = \mathcal{F}
\end{align}
with $\mathcal{M}$ and $\mathcal{N}$ appropriately defined from above. 
To obtain existence of a solution, we recall that the matrix pencil $s\mathcal{M} + \mathcal{N}$ should not degenerate for some $s \neq 0$. It is sufficient to demand that for $s = 1$,  $\mathcal{M}+\mathcal{N}$ is invertible or in other words, $(\mathcal{M} + \mathcal{N}) \boldsymbol \Xi = 0$  has only $\Xi = 0$ as the only solution. This is verified by showing the system below has zero as its only solution.
\begin{multline}
\noindent \text{For } \boldsymbol{w} \in \boldsymbol V_m\\
 ( {{\boldsymbol u}_m}, {\boldsymbol w}) + 2 G \big(\varepsilon({\boldsymbol u}_m), \varepsilon({\boldsymbol w})\big) + \lambda \big(\nabla \cdot {\boldsymbol u}_m,\nabla\cdot {\boldsymbol w}\big)  - \alpha \big( p_m ,\nabla\cdot {\boldsymbol w}\big) \\
 +\big(p_{cm} ,[{\boldsymbol w}]_{\CS} \cdot {\boldsymbol n}^+\big)_\CS = 0
\label{eq:varumixTD1}    
\end{multline} 
for $\theta \in Q_m$, 
\begin{align}
 \nonumber  \Big(\Big(\Big(\frac{1}{M} + c_f \varphi_0\Big) p_m +\alpha  \nabla\cdot {\boldsymbol u}_m \Big), \theta \Big)
+ \frac{1}{\mu_f}\big( \boldsymbol K \nabla p_m ,\nabla \theta_i \big) + (c_{fc}  p_{cm}, \theta)_{\CS} \\+   (-[{\boldsymbol u}_m]_\CS \cdot {\boldsymbol n}^+, \theta)_\CS + 
 \frac{1}{12 \mu_f} ( w^{3/2} (\overline{\nabla} p_{cm}), w^{3/2} \overline{\nabla} \theta )_\CS =  0.
\label{eq:varpmixTD3}
\end{align}
Testing $\boldsymbol w = \uu_m, \theta = p_m$ and adding the two equations, we see that the coupling terms cancel leaving positive terms on the left hand side. This proves the result.

\subsection{A priori estimates and convergence}

To establish convergence of the finite dimensional solution to that of the continuous one, we use compactness arguments after deriving a priori estimates. With the finite dimensional solution in hand, we can use them as the test functions. We begin with the variational form \eqref{eq:varumix}-\eqref{eq:varpmix2} and choose $\boldsymbol v = \partial_t \boldsymbol u_m, \theta =  p_m, \theta_c = p_{cm}$, to obtain
\begin{multline} \label{eq:apestimate1}
(\partial_{tt} {\boldsymbol u_m}, \partial_t{\boldsymbol u_m})
+ 2 G \big(\varepsilon({\boldsymbol u_m}), \varepsilon({\partial_t \boldsymbol u_m})\big) + \lambda \big(\nabla \cdot { \boldsymbol u_m},\nabla\cdot {\partial_t \boldsymbol u_m}\big) 
- \alpha \big( p_m ,\nabla\cdot {\partial_t \boldsymbol u_m}\big)
\\[0.25cm]
 + \big(p_{cm} ,[{\partial_t \boldsymbol u_m}]_{\CS} \cdot {\boldsymbol n}^+\big)_\CS
  + \Big({\partial_t}\Big(\Big(\frac{1}{M} + c_f \varphi_0\Big) p_m +\alpha  \nabla\cdot {\boldsymbol u_m} \Big) , p_m \Big) 
+ \frac{1}{\mu_f}\big( K \nabla p_m,\nabla p_m \big) 
\\
+ (c_{fc} \partial_t p_{cm}, p_{cm}) -  (\partial_t[{\boldsymbol u_m}]_\CS \cdot {\boldsymbol n}^+, p_{cm} )_\CS 
+ \frac{1}{12 \mu_f} \langle w^{\frac{3}{2}} (\overline{\nabla} p_{cm} - \rho_{f,r} g \eta), w^{\frac{3}{2}} \overline{\nabla} p_{cm} \rangle  \\=  \langle\tilde q_W , p_{cm} \rangle_\CS,
- \langle {q}_l, p_{cm}\rangle_\CS + \big(\tilde q, p_m \big) + \big( {\boldsymbol f} ,{\boldsymbol u_m}\big).
\end{multline}
Following estimates are straightforward and will be used for the a priori estimates
\begin{eqnarray*}
  (\partial_{tt} {\boldsymbol u_m}, \partial_t{\boldsymbol u_m}) &=& \dfrac{1}{2}\dfrac{d}{dt}\|\partial_t \uu_m\|^2,
\\ 
  2 G \big(\varepsilon({\boldsymbol u_m}), \varepsilon({\partial_t \boldsymbol u_m})\big) &=& 2G \dfrac{1}{2}\dfrac{d}{dt}\|\varepsilon(\uu_m) \|^2,
\\
  \lambda \big(\nabla \cdot {\boldsymbol u_m},\nabla\cdot {\partial_t \boldsymbol u_m}\big) &=& \lambda \dfrac{1}{2}\dfrac{d}{dt}\|\nabla \cdot \uu_m \|^2,
\\
  \Big(\frac{\partial }{\partial t}\Big(\Big(\frac{1}{M} + c_f \varphi_0\Big) p_m \Big) , p_m \Big) &=& \Big(\frac{1}{M} + c_f \varphi_0 \Big ) \dfrac{1}{2}\dfrac{d}{dt}\| p_m \|^2, 
\\
  \Big(\frac{\partial }{\partial t}\big(\alpha  \nabla \cdot {\uu_m} \big) , p_m \Big)  - \alpha \big( p_m ,\nabla\cdot {\partial_t \boldsymbol u_m}\big) &=& 0,
\\
  -\big(p_{cm} ,[{\partial_t \boldsymbol u_m}]_{\CS} \cdot {\boldsymbol n}^+\big)_\CS + \Big( \dfrac{d}{dt} [{\boldsymbol u_m}]_\CS \cdot {\boldsymbol n}^+, p_{cm} \Big)_\CS &=& 0,
\\
  \big (c_{fc} \partial_t p_{cm}, p_{cm} \big ) &=& c_{fc} \dfrac{1}{2}\dfrac{d}{dt}\| p_{cm} \|^2,
\\
  \frac{1}{\mu_f}\big( K \nabla p_m,\nabla p_m \big) &=& \frac{1}{\mu_f} \| K^{1/2} \nabla p_m \|^2,
\\
  \frac{1}{12 \mu_f} \langle w^{\frac{3}{2}} (\overline{\nabla} p_{cm} , w^{\frac{3}{2}} \overline{\nabla} p_{cm} \rangle &=& \frac{1}{12 \mu_f} \| w^{\frac{3}{2}} \overline{\nabla} p_{cm} \|^2.
\end{eqnarray*}
Using Young's inequality, with the coefficients $\delta_i, i = 1,2,\ldots$ to be chosen appropriately
\begin{eqnarray*}  
  \big \langle - \rho_{f,r} g \eta, w^{\frac{3}{2}} \overline{\nabla} p_{cm} \rangle &\leq& \delta_1 \|w^{\frac{3}{2}} \nabla p_{cm}\|^2 + \frac{1}{2\delta_1} \rho_{f,r} |g| \|\eta\|^2,
\\
  \langle\tilde q_W , p_{cm} \rangle_\CS &\leq& \frac{1}{2\delta_2}\|q_W\| + \delta_2 \|p_{cm}\|^2,
\\
  \big(\tilde q, p_m \big) &\leq& \frac{1}{2\delta_4}\|q\|^2 + \delta_4 \|p_m\|^2,
\\
  \big( {\boldsymbol f} ,{\boldsymbol u_m}\big) &\leq& \frac{1}{2\delta_5}\|\boldsymbol f\|^2 + \delta_5 \|\uu_m\|^2 .
\end{eqnarray*}
 With these estimates, we rewrite \eqref{eq:apestimate1} as
  \begin{multline*}
   \dfrac{d}{dt}\|\partial_t \uu_m\|^2 + 2G \dfrac{d}{dt}\|\varepsilon(\uu_m) \|^2 
+ \lambda \dfrac{d}{dt}\|\nabla \cdot \uu_m \|^2  
+ \big((\frac{1}{M} + c_f \varphi_0) \dfrac{d}{dt}\|\nabla p_m \|^2 \\
 + c_{fc} \dfrac{d}{dt}\| p_{cm} \|^2 + \frac{1}{\mu_f} \| K^{1/2} \nabla p_m \|^2 
+ \frac{1}{12 \mu_f} \| w^{\frac{3}{2}} \overline{\nabla} p_{cm} \|^2 
+  \delta_1 \|w^{\frac{3}{2}} \nabla p_{cm}\|^2 + \frac{1}{2\delta_1} \rho_{f,r} |g| \|\eta\|^2 \\
 \leq  C \Big( \frac{1}{2\delta_2}\|q_W\| 
 + \frac{1}{2\delta_3}\|q_l\|^2 
  + \frac{1}{2\delta_4}\|q\|^2 
   + \frac{1}{2\delta_5}\|\boldsymbol f\|^2 \Big)
\end{multline*}
Integrating in time over $(0,t)$, with appropriately chosen $\delta_i, i = 1,\ldots,5$, we get the following proposition.

\begin{proposition} \label{pro:estimates}
Let $\boldsymbol f \in H^2(0,T; L^2(\Omega)^d)$, $\tilde{q} \in L^2(\Omega \times (0,T))$, $\tilde{q}_W \in H^1(0,T;L^2(\CS))$, $p_0 \in Q$, and $w$ satisfy \eqref{eq:hyp_w}, then the solution $\uu_m, p_m$ satisfying \eqref{eq:varumix}-\eqref{eq:varpmix2} satisfies the following energy estimate, for $t \leq T$
\begin{multline}
  \|\partial_t \uu_m(t)\|^2 + G \|\varepsilon(\uu_m)(t) \|^2 
+ \lambda \|\nabla \cdot \uu_m(t) \|^2  
+ \Big(\frac{1}{M} + c_f \varphi_0\Big) \dfrac{1}{2}\| p_m(t) \|^2 \\
 + c_{fc} \| p_{cm}(t) \|^2 + \frac{1}{\mu_f} \int_0^t \| K^{1/2} \nabla p_m(\tau) \|^2 d\tau
+ \frac{1}{12 \mu_f} \int_0^t \| w^{\frac{3}{2}} \overline{\nabla} p_{cm}(\tau) \|^2 d\tau \\[0.15cm]
\leq C  \left ( \rho_{f,r} |g| \|\eta\|^2 + \|q_W\| +  \|q_l\|^2 + \|q\|^2 +\|\boldsymbol f\|^2 \right).
\end{multline}
Here, $C$ depends on $G,\lambda, \boldsymbol K, \mu_f, c_f, M$ and is independent of dimension $m$ and solutions $\uu_m, p_m$.
\end{proposition}

\begin{remark}
We point out the cancellation of coupling terms in the a priori estimate above. We observe that the fluid pressure term $-\alpha (p ,\nabla\cdot {\boldsymbol v})$ in the mechanics equation gets cancelled with the $(\partial_t (\alpha  \nabla\cdot {\boldsymbol u}), \theta)$ term in the flow equation for the choice of test functions $\boldsymbol v = \partial_t \uu$, $\theta = p$. Similarly, the fracture pressure term $-\big(p_c ,[{\boldsymbol v}]_{\CS} \cdot {\boldsymbol n}^+\big)_\CS$ in the mechanics equation gets cancelled with the $ (\partial_t[{\boldsymbol u}]_\CS \cdot {\boldsymbol n}^+, \theta_c )_\CS$ term for the test functions $\boldsymbol v = \partial_t \uu, \theta_c = p_c$. 
\end{remark}

The above estimate implies the following convergence result, if needed, up to a subsequence,
\begin{align*}
 p_m &\rightarrow p\ \text{ weakly in } L^\infty(0,t; L^2(\Omega)) \cap L^2(0,t; Q),\\  
 %p_{cm} &\rightarrow p_c \text{ weakly in } L^\infty(0,t; H^1(\CS))\\
 \uu_m & \rightarrow \uu \text{ weakly in } L^\infty(0,t; \V) \cap H^1(0,t; L^{2}(\Omega)).
\end{align*}
The existence of the solution follows by simply taking the limits in
\begin{equation} \text{for all } {\boldsymbol v} \in \V\,,\,(\partial_{tt} {\boldsymbol u_m}, {\boldsymbol v}) + 2 G \big(\varepsilon({\boldsymbol u_m}), \varepsilon({\boldsymbol v})\big) + \lambda \big(\nabla \cdot {\boldsymbol u_m},\nabla\cdot {\boldsymbol v}\big) 
- \alpha \big( p_m ,\nabla\cdot {\boldsymbol v}\big)
   +\big(p_{cm} ,[{\boldsymbol v}]_{\CS} \cdot {\boldsymbol n}^+\big)_\CS = \big( {\boldsymbol f} ,{\boldsymbol v}\big)\,,\,\end{equation}
\begin{align}
   \text{for all } \theta \in Q\,,\,\Big({\partial_t}\Big(\Big(\frac{1}{M} + c_f \varphi_0\Big) p_m +\alpha  \nabla\cdot {\boldsymbol u_m} \Big) ,\theta \Big)+ \frac{1}{\mu_f}\Big( K \nabla p_m,\nabla \theta \Big) = (\tilde q, \theta),
\\
 (c_{fc} \partial_t p_{cm}, \theta_c) -   (\partial_t[{\boldsymbol u_m}]_\CS \cdot {\boldsymbol n}^+, \theta_c )_\CS + \frac{1}{12 \mu_f} \langle w^{\frac{3}{2}} (\overline{\nabla} p_{cm} - \rho_{f,r} g \eta), w^{\frac{3}{2}} \overline{\nabla} \theta_c \rangle =  \langle\tilde q_W , \theta_c \rangle_\CS,
- \langle {q}_l, \theta_c \rangle_\CS,
\end{align}
with $\theta_c$ being the trace of $\theta$. 
We then obtain the variational form:
\begin{equation} \text{for all } {\boldsymbol v} \in \V\,,\,(\partial_{tt} {\boldsymbol u}, {\boldsymbol v}) + 2 G \big(\varepsilon({\boldsymbol u}), \varepsilon({\boldsymbol v})\big) + \lambda \big(\nabla \cdot {\boldsymbol u},\nabla\cdot {\boldsymbol v}\big) 
- \alpha \big( p ,\nabla\cdot {\boldsymbol v}\big)
   +\big(p_c ,[{\boldsymbol v}]_{\CS} \cdot {\boldsymbol n}^+\big)_\CS = \big( {\boldsymbol f} ,{\boldsymbol v}\big)\,,\, 
\end{equation}
\begin{align}
   \text{for all } \theta \in Q\,,\,\Big(\frac{\partial }{\partial t}\Big(\Big(\frac{1}{M} + c_f \varphi_0\Big) p +\alpha  \nabla\cdot {\boldsymbol u} \Big) ,\theta \Big)
+ \frac{1}{\mu_f}\big( K \nabla p,\nabla \theta \big) = \big(\tilde q, \theta \big),\\
\,,\, (c_{fc} \partial_t p_c, \theta_c) -   (\partial_t[{\boldsymbol u}]_\CS \cdot {\boldsymbol n}^+, \theta_c )_\CS + \frac{1}{12 \mu_f} \langle w^{\frac{3}{2}} (\overline{\nabla} p_c - \rho_{f,r} g \eta), w^{\frac{3}{2}} \overline{\nabla} \theta_c \rangle =  \langle\tilde q_W , \theta_c \rangle_\CS,
- \langle {q}_l, \theta_c\rangle_\CS,
\end{align}
with $\theta_c$ being the trace of $\theta$.
We summarize the existence result in the following  
\begin{theorem}
There exists a limit triple $(p,p_c, \uu)$ that satisfies the variational equation in formulation 1. 
\end{theorem}
\subsection{Uniqueness of solution for Problem ($Q_0$)}
After establishing the existence result, we proceed to the uniqueness statement. The uniqueness is a straightforward consequence of the linearity of the model equation and the estimates obtained in Proposition \ref{pro:estimates}. 
\begin{theorem}
There exists one and only one solution satisfying the variational equation in formulation 1. 
\end{theorem}
\begin{proof}
The proof proceeds via contradiction. If possible, let there be two solutions $(p_1,p_{c1}, \uu_1)$ and $(p_2,p_{c2}, \uu_2)$ satisfying the variational equation in formulation 1. We define, 
\begin{align*}
p = p_1 - p_2, \quad p_c = p_{c1} - p_{c2}, \quad \uu = \uu_1 - \uu_2.
\end{align*}
The variational equation satisfied by $(p, p_c, \uu)$ is given by: find  $\boldsymbol u \in L^\infty(0,T;\V)$, $\partial_t{\boldsymbol u} \in L^\infty(0,T;L^2(\Omega)) $, $\partial_{tt} \boldsymbol u \in L^2(0,T;\V^{'})$,  $p \in L^\infty(0,T;H^1(\Omega))$, $p_c \in L^2(0,T;H^{\frac{1}{2}}(\CS))$, such that 
\begin{equation} 
\text{for all } {\boldsymbol v} \in \V\,,\,(\partial_{tt} {\boldsymbol u}, {\boldsymbol v}) + 2 G \big(\varepsilon({\boldsymbol u}), \varepsilon({\boldsymbol v})\big) + \lambda \big(\nabla \cdot {\boldsymbol u},\nabla\cdot {\boldsymbol v}\big) 
- \alpha \big( p ,\nabla\cdot {\boldsymbol v}\big)
   +\big(p_c ,[{\boldsymbol v}]_{\CS} \cdot {\boldsymbol n}^+\big)_\CS = 0 
\label{eq:varumix3}
\end{equation}
and
\begin{align}
\label{eq:varpmix12}   \text{for all } \theta \in H^1(\Omega)\,,\,\Big({\partial_t}\Big(\Big(\frac{1}{M} + c_f \varphi_0\Big) p +\alpha  \nabla\cdot {\boldsymbol u} \Big), \theta\Big)
+ \frac{1}{\mu_f}\big( K \nabla p,\nabla \theta \big) =  \langle \tilde{q}_L,\theta \rangle_{\CS},\\
(c_{fc} \partial_t p_c, \theta_c)_{\CS} +   (\partial_t[{\boldsymbol u}]_\CS \cdot {\boldsymbol n}^+, \theta_c )_\CS + \frac{1}{12 \mu_f} ( w^{\frac 3 2} (\overline{\nabla} p_c - \rho_{f,r} g \nabla \eta), w^{\frac 3 2} \overline{\nabla} \theta_c )_\CS =  
- \langle \tilde{q}_L, \theta_c \rangle_\CS ,
\label{eq:varpmix41}
\end{align}
 subject to the initial conditions, 
\begin{align}
\label{eqn:initial21}
\Big(\Big(\frac{1}{M} + c_f \varphi_0\Big)p + \alpha \nabla \cdot {\boldsymbol u}\Big)\Big|_{t=0} = 0
\end{align}
and $\uu(x,0) = \uu_0, \partial_t \uu(x,0) = \uu_1$.
Moreover, the boundary condition is also homogeneous, Dirichlet for $p, \uu$. In other words, there is no forcing term in the mechanics and flow equations, and the initial and boundary conditions all vanish. 
The estimate as in Proposition \ref{pro:estimates} takes the form
 \begin{multline} \label{eq:aprioriestimates2}
\dfrac{1}{2}\|\partial_t \uu(t)\|^2 + G\|\varepsilon(\uu)(t) \|^2 
+ \lambda \dfrac{1}{2}\|\nabla \cdot \uu(t) \|^2  
+ \left(\frac{1}{M} + c_f \varphi_0)\right) \dfrac{1}{2}\|p(t) \|^2 \\
 + c_{fc} \dfrac{1}{2}\| p_{c}(t) \|^2 + \frac{1}{\mu_f} \int_0^t \| K^{1/2} \nabla p(\tau) \|^2 d\tau
+ \frac{1}{12 \mu_f} \int_0^t \| w^{\frac{3}{2}} \overline{\nabla} p_{c}(\tau) \|^2 d\tau  \leq 0.
   \end{multline}
This implies, $\partial_t \uu \equiv 0, \nabla \uu \equiv 0, \nabla p \equiv 0$. For $\uu$ we have used Korn's inequality. Using the homogeneous Dirichlet boundary condition and the Poincar\'{e} estimate, this implies $\uu \equiv 0, p \equiv 0$; noting that $p_c$ is the trace of $p$, uniqueness of solution triple follows. 
\end{proof}
\section{Including normal compliance and friction: Problem (Q)} \label{sec:contact2}
As mentioned before, the fracture flow model contains the width of the fracture both in the time derivative term and in the permeability description. However, the model equations have no means of ensuring that the width remains nonnegative. This means that the fracture surfaces may keep penetrating into each other since there is no mechanism to stop this penetration. This is clearly unphysical. Motivated by the arguments from tribology, we will append the existing model by ensuring normal compliance using simple phenomenological law for the fracture surface that employs the normal stiffness of the interface. 
% Recall on the fracture surface $\CS$, instead of \eqref{eq:tildsig2} we impose,
% \begin{align}
%  \label{eq:tildsig3} \left(\boldsymbol \sigma^{\rm{por}} \boldsymbol n^* \right) \cdot \boldsymbol n^{*} = -p_c - c_n[([\boldsymbol u]_\CS - g_0]_{+}^{m_n}.   
% \end{align}
% Here, $g_0$ is the width of the fracture in the undeformed configuration and $[\cdot]_{+}$ is the positive cut of the function; $c_n$ is a non-negative constant, $1\leq m_n \leq 3$ for $d=3$ and $m_n < \infty$ for $d = 2$. 
% The above modification in the interface condition imposes a resistance force when the fracture surfaces mutually penetrate. The $c_n, m_n$ are constants that characterize the material interface parameters. The restriction on $m_n$ is motivated by the fact that for $m_n+1$, $H^{1/2}(\CS)$ is continuously embedded in $L^{m_n+1}(\CS)$. This will be useful in obtaining the compactness results later. The above model is adapted from the classical work of Martins \& Oden \cite{martins1987existence} where a similar model is developed in the context of contact problems in mechanics. \\
Recall the nonlinear map, $P_n: V \rightarrow V^\prime$, such that for $\boldsymbol u, \boldsymbol v \in V$,
\begin{align}
    \langle P_n(\boldsymbol u), \boldsymbol v \rangle = \int_{\CS}  c_n \Big (-[\boldsymbol u]_\CS \cdot \boldsymbol n^{+} - g_0 \Big)_{+}^{m_n} ([-{\boldsymbol v}]_\CS \cdot {\boldsymbol n}^+) ds. 
\end{align}
Next, we consider the case when friction is included in the contact mechanics.

% The following conditions are used to ensure this:
% \[  \left(\boldsymbol \sigma^{\rm{por}} \boldsymbol n^* \right) \cdot \boldsymbol n^{*} = -p_c - c_n[([\boldsymbol u]_\CS - g_0]_{+}^{m_n};  \]
%   $  \text{ if } [\uu]_\CS  \leq g_0  \text{ then }  $
%   \[    |\sigma_T(\uu)| = 0
%  \]
%   $  \text{ if } [\uu]_\CS >g_0  \text{ then }  $
%   \[    |\sigma_T(\uu)| = c_T[([\uu]_\CS - g_0)_{+}]^{m_T}  \text{ and there exists } \lambda \geq 0, \dot{\uu}_T  = -\lambda \sigma_T(\uu). 
%  \]
% The above constraint reflects slip and is a generalization of Coulomb's law. It does not model the onset of slip inclusion of which leads to a variational inequality and can be easily handled. 

We recall 
\[j(\uu, \vv) = \int_{\CS} c_T \Big(-[\uu]_{\CS} \cdot \boldsymbol n^{+} - g_0 \Big)_{+}^{m_T}|\vv_T|ds, \quad \uu, \vv \in \V . \]
We also let $\partial_2 j(\uu, \partial_t \uu) \in \V \times \V$  denote the partial subdifferential of $j$ with respect to the second argument.

\medskip\medskip

\noindent
Here, note that in addition to the linear elasticity part, we include viscoelastic contribution as well with a viscous damping term characterized by $\gamma>0$. This is needed to control the time derivative of the  displacement $\uu$ terms on the boundary $\CS$ using the gradient of time derivative in the bulk domain given by the viscoelastic term.\\  

\subsection{Weak formulation of Problem (Q): Variational inequality and flow equations}
 
The conformal variational formulation of Problem (Q) reads: For given ${\boldsymbol f} \in L^2((\Omega \setminus \CS) \times ]0,T[)^d$, $\tilde q \in L^2(\Omega \times ]0,T[)$, and  $\tilde q_L \in L^2(0,T;H^{1}(\CS)')$, find  $\boldsymbol u \in L^\infty(0,T;\V)$, $\partial_t{\boldsymbol u} \in L^\infty(0,T;L^2(\Omega)) $, $\partial_{tt} \boldsymbol u \in L^2(0,T;\V^{'})$,  $p \in L^\infty(0,T;H^1(\Omega))$, $p_c \in L^2(0,T;H^{1/2}(\CS))$, such that 
 for all $ {\boldsymbol v} \in \V$,
\begin{multline}
 (\partial_{tt} {\boldsymbol u}, {\boldsymbol v} - \partial_t{\uu}) + 2 G \big(\varepsilon({\boldsymbol u}), \varepsilon({\boldsymbol v} - \partial_t{\uu})\big) + \lambda \big(\nabla \cdot {\boldsymbol u},\nabla\cdot ({\boldsymbol v} - \partial_t{\uu}) \big) + \gamma (\varepsilon (\partial_t{\uu}), \varepsilon(\vv - \partial_t{\uu}) )\\ + \gamma (\nabla \cdot \partial_t{\uu}, \nabla (\vv - \partial_t{\uu}))
- \alpha \big( p ,\nabla\cdot ({\boldsymbol v} - \partial_t{\uu}) \big) 
 +\big(p_c ,[{\boldsymbol v} - \partial_t \uu]_{\CS} \cdot {\boldsymbol n}^+\big)_\CS + \langle P_n([\uu]_\CS), \vv - \partial_t{\uu} \rangle \\
 +  \langle j(\uu(t), {\vv}) - j(\uu, \partial_t{\uu} - \vv) \rangle \geq \big( {\boldsymbol f} ,{\boldsymbol v} - \partial_t{\uu} \big) ,
\label{eq:varumixf1}    
\end{multline} 
for all $\theta \in H^1(\Omega)\,,$
\begin{align}
\label{eq:varpmixf2}    \Big({\partial_t}\Big(\Big(\frac{1}{M} + c_f \varphi_0\Big) p +\alpha  \nabla\cdot {\boldsymbol u} \Big), \theta\Big)
+ \frac{1}{\mu_f}\big( \boldsymbol K \nabla p,\nabla \theta \big) &= \big(\tilde q, \theta \big) + \langle \tilde{q}_L,\theta \rangle_{\CS} ,
\end{align}
and for all $\theta_c \in H_w^1(\CS)$,
\begin{multline}
(c_{fc} \partial_t p_c, \theta_c)_{\CS} -  (\partial_t[{\boldsymbol u}]_\CS \cdot {\boldsymbol n}^+, \theta_c )_\CS
\\
%\big(p_c ,[{\boldsymbol v}]_{\CS} \cdot {\boldsymbol n}^+\big)_\CS,  \dfrac{d}{dt}  ([{\boldsymbol u}]_\CS \cdot {\boldsymbol n}^+, \theta_c )_\CS, \Big (\partial_t (\alpha  \nabla\cdot {\boldsymbol u} \Big), \theta\Big), \alpha \Big( p ,\nabla\cdot ({\boldsymbol v} - \dot{\uu}(t)) \Big) 
+ \frac{1}{12 \mu_f} ( w^{3/2} (\overline{\nabla} p_c - \rho_{f,r} g \nabla \eta), w^{3/2} \overline{\nabla} \theta_c )_\CS =  ( \tilde q_W , \theta_c )_\CS,
- \langle \tilde{q}_L, \theta_c \rangle_\CS ,
\label{eq:varpmixf3}
\end{multline}
 subject to the initial condition,
\begin{align}
\label{eqn:initialf1}
\Big(\Big(\frac{1}{M} + c_f \varphi_0\Big)p + \alpha \nabla \cdot {\boldsymbol u}\Big)\Big|_{t=0} = \Big(\frac{1}{M} + c_f \varphi_0\Big)p_0 + \alpha \nabla \cdot {\boldsymbol u}_0
\end{align}
in addition to the initial conditions for displacement $\uu(x, 0) = \uu_0, \partial_t \uu(x,0) = \uu_1.$
This weak formulation is obtained starting from problem (Q) using the standard method of taking smooth test functions and then performing partial integration and using the interface  boundary conditions. The difference with the problem $(Q_0)$ is in the first equation. Here, in contrast the contact mechanics description leads to a variational inequality. The inequality is obtained due to the friction condition on $\CS$ implying for $\vv \in \boldsymbol V$
\[ \boldsymbol \sigma_T \cdot (\vv_T - \partial_t \uu_T) + c_T[(-[\uu]_\CS \cdot \boldsymbol n^{+}-g_0)_{+}]^{m_T}(|\vv_T| + |\partial_t\uu|) \geq 0. \]

%\subsection{Existence of a solution {\color{red} [can't we remove this subsection or merge it with Section 6.4?]}} \label{sec:wellposedness2}

\subsection{Uniqueness of solution}

We let $(\uu_1,p_1,p_{c1})$ and $(\uu_2,p_2,p_{c2})$ be two solution triples solving \eqref{eq:varumixf1}-\eqref{eqn:initialf1}. We take \eqref{eq:varumixf1}, consider the equation for $\uu_1$ and choose $\vv = \partial_t \uu_2$. Similarly, considering the equation for $\uu_2$ we choose $\vv = \partial_t \uu_1$. That is, at a.e. $t$, we have 
\begin{multline}
 (\partial_{tt} {\boldsymbol u}_1 , \partial_t{\uu_2} - \partial_t{\uu}_1) + 2 G \big(\varepsilon({\uu}_1), \varepsilon(\partial_t{\uu_2} - \partial_t{\uu}_1)\big) + \lambda \big(\nabla \cdot {\boldsymbol u}_1,\nabla\cdot (\partial_t{\uu_2} - \partial_t{\uu}_1) \big) + \gamma \big(\nabla \cdot \partial_t{\boldsymbol u}_1,\nabla\cdot (\partial_t{\uu_2} - \partial_t{\uu}_1) \big)\\ + \gamma (\varepsilon (\partial_t{\uu}_1), \varepsilon(\partial_t{\uu_2} - \partial_t{\uu}_1)) 
- \alpha \big( p_1 ,\nabla\cdot (\partial_t{\uu_2} - \partial_t{\uu}_1) \big) 
 +\big(p_{c1} ,[\partial_t{\uu_2} - \partial_t \uu_1]_{\CS} \cdot {\boldsymbol n}^+\big)_\CS \\
 + \langle P_n([\uu_1]_\CS), \partial_t{\uu_2} - \partial_t{\uu}_1 \rangle 
 +  \langle j(\uu_1, \partial_t{\uu_2}) - j(\uu_1, \partial_t{\uu_1} - \partial_t{\uu_2}) \rangle \geq \big( {\boldsymbol f} ,\partial_t{\uu_2} - \partial_t{\uu_1} \big) 
\label{eq:varumixfu1}    
\end{multline} 
and
\begin{multline}
 (\partial_{tt} {\boldsymbol u}_2 , \partial_t{\uu_1} - \partial_t{\uu}_2) + 2 G \big(\varepsilon({\uu}_2), \varepsilon(\partial_t{\uu_1} - \partial_t{\uu}_2)\big) + \lambda \big(\nabla \cdot {\uu}_2,\nabla\cdot (\partial_t{\uu_1} - \partial_t{\uu}_2) \big) + \gamma \big(\nabla \cdot \partial_t{\uu}_2,\nabla\cdot (\partial_t{\uu_1} - \partial_t{\uu}_2) \big)\\ + \gamma (\partial_t{\uu}_2, \partial_t{\uu_1} - \partial_t{\uu}_2) 
- \alpha \big( p_2 ,\nabla\cdot (\partial_t{\uu_1} - \partial_t{\uu}_2) \big) 
 +\big(p_{c2} ,[\partial_t{\uu_1} - \partial_t \uu_2]_{\CS} \cdot {\boldsymbol n}^+\big)_\CS \\
 + \langle P_n([\uu_2]_\CS), \partial_t{\uu_1} - \partial_t{\uu}_2(t) \rangle 
 +  \langle j(\uu_2, \partial_t{\uu_1}) - j(\uu_2, \partial_t{\uu_2} - \partial_t{\uu_1}) \rangle \geq \big( {\boldsymbol f} ,\partial_t{\uu_1} - \partial_t{\uu_2} \big) .
\label{eq:varumixfu4}    
\end{multline} 
Subtracting the equations for $p_1$ and $p_2$ in \eqref{eq:varpmixf2} and using $\theta  = p_1-p_2$, we obtain
\begin{align}
\label{eq:varpmixfu2}    \Big({\partial_t}\Big(\Big(\frac{1}{M} + c_f \varphi_0\Big) (p_1 - p_2) +\alpha  \nabla\cdot \Big ( {\uu}_1 - {\uu}_2  \Big),  p_1 - p_2 \Big)
\\ \nonumber \noindent + \frac{1}{\mu_f}\big( \boldsymbol K (\nabla p_1 - \nabla p_2), (\nabla p_1 - \nabla p_2) \big) &= \langle \tilde{q}_{L1} - \tilde{q}_{L2}, (p_1 - p_2) \rangle_{\CS} ,
\end{align}
and subtracting the equations for $p_{c1}$ and $p_{c2}$ (cf.~\eqref{eq:varpmixf3}) using $\theta_c = p_{c1} - p_{c2}$, we get
\begin{multline}
(c_{fc} \partial_t (p_{c1} - p_{c2}), p_{c1} - p_{c2})_{\CS} +  (\partial_t[{\uu}_1 - \uu_2]_\CS \cdot {\boldsymbol n}^+, p_{c1} - p_{c2})_\CS
\\
+ \frac{1}{12 \mu_f} ( w^{\frac 3 2} (\overline{\nabla} (p_{c1} - p_{c2}), w^{\frac 3 2} \overline{\nabla} p_{c1} - p_{c2})_\CS =  
- \langle \tilde{q}_{L1}- \tilde{q}_{L2}, p_{c1} - p_{c2} \rangle_\CS.
\label{eq:varpmixfu3}
\end{multline}
Adding the above three equations \eqref{eq:varumixfu4}-\eqref{eq:varpmixfu3}, we obtain (recalling that the trace of $p_i, i = 1,2$ on $\mathcal{C}$ is $p_{ci}, i = 1,2$ in cancelling the { right-hand sides} of \eqref{eq:varpmixfu2}-\eqref{eq:varpmixfu3}), 
\begin{multline}
  \dfrac{1}{2}\dfrac{d}{dt}\|\partial_t{\uu_1} - \partial_t{\uu}_2\|^2 + \dfrac{d}{dt} G\|\varepsilon ({\uu_1} - {\uu}_2) \|^2 +\dfrac{1}{2} \lambda \dfrac{d}{dt}\|\nabla\cdot ({\uu_1} - {\uu}_2) \|^2 \\
  + \gamma \| \nabla \big(\partial_t{\uu_1} - \partial_t{\uu}_2 \big)\|^2 +c_{fc}  \dfrac{1}{2} \dfrac{d}{dt}\|p_{c1} - p_{c2}\|_{\CS}^2 + \dfrac{1}{2} \dfrac{d}{dt} \Big\|\Big(\frac{1}{M} + c_f \varphi_0\Big) (p_1 - p_2)  \Big\|^2
  \\
 + \Big\|\frac{1}{12 \mu_f} ( w^{\frac 3 2} (\overline{\nabla} (p_{c1} - p_{c2})) \Big\|_{\CS}^2 +   \frac{1}{\mu_f} \Big \|\big( \boldsymbol K (\nabla p_1 - \nabla p_2) \big)\Big\|^2
 \\
  \leq -\langle P_n([\uu_1]_\CS) - P_n([\uu_2]_\CS), \partial_t{\uu_1} - \partial_t{\uu}_2 \rangle   +   j(\uu_1, \partial_t{\uu_2}) - j(\uu_1, \partial_t{\uu_1})  
  \\[0.25cm]
 +   j(\uu_2, \partial_t{\uu_1}) - j(\uu_2, \partial_t{\uu_2}) . 
%\label{eq:varumixfu1}    
\end{multline} 
We note that $p_1 - p_2$, $p_{c1}-p_{c2}$ and $\uu_1 - \uu_2$ and have zero initial and boundary conditions. Integration in time from $0$ to $t$ gives 
\begin{multline}
  \dfrac{1}{2}\|\partial_t{\uu_1} - \partial_t{\uu}_2\|^2 +  G\|\varepsilon({\uu_1} - {\uu}_2) \|^2 +\dfrac{1}{2} \lambda \|\big(\nabla\cdot ({\uu_1} - {\uu}_2) \big)\|^2 
  \\
  + \gamma \int_0^t \| \nabla \big(\partial_t{\uu_1} - \partial_t{\uu}_2 \big)\|^2 d\tau +c_{fc}  \dfrac{1}{2} \|(p_{c1} - p_{c2})\|_{\CS}^2 + \dfrac{1}{2}  \Big\|\Big(\Big(\frac{1}{M} + c_f \varphi_0\Big) (p_1 - p_2)  \Big) \Big\|^2
  \\
 + \int_0^t \Big\|\frac{1}{12 \mu_f} ( w^{3/2} (\overline{\nabla} (p_{c1} - p_{c2})\|_{\CS}^2 d\tau +   \frac{1}{\mu_f} \int_0^t \|\big( \boldsymbol K (\nabla p_1 - \nabla p_2)\big)\Big\|^2 d\tau
 \\[0.25cm]
  \leq -\langle P_n([\uu_1]_\CS) - P_n([\uu_2]_\CS), \partial_t{\uu_1} - \partial_t{\uu}_2 \rangle   +   j(\uu_1, \partial_t{\uu_2}) - j(\uu_1, \partial_t{\uu_1})  
  \\[0.25cm]
 +   j(\uu_2, \partial_t{\uu_1}) - j(\uu_2, \partial_t{\uu_2}) . 
%\label{eq:varumixfu1}    
\end{multline} 
Next, we estimate the right-hand side of the above inequality. The normal contact term is estimated as 
\begin{multline}
 \int_0^t |\langle P_n([\uu_1]_\CS) - P_n([\uu_2]_\CS), \partial_t{\uu_1} - \partial_t{\uu}_2 \rangle| d\tau
\\  
  \leq \displaystyle \int_0^t \displaystyle \int_{\CS} c_n|(-[\uu_2]_\CS \cdot \boldsymbol n^{+} - g_0)_{+})^{m_n} - ([\uu_1]_\CS \cdot \boldsymbol n^{+} - g_0)_{+})^{m_n}| \, |\partial_t{\uu_1} - \partial_t{\uu}_2| ds d\tau
  \\  
  \leq C \int_0^t\|\nabla \uu_1 - \nabla \uu_2\| \|\nabla \partial_t\uu_1 - \nabla \partial_t\uu_2\| d\tau
\end{multline}
where we have used the Lipschitz continuity of $P_n$ and trace estimate for $\partial_t (\uu_1-\uu_2)$. The Lipschitz continuity is based on the a priori estimate that will be obtained later. The friction term is similarly treated, using that
\begin{multline}
  j(\uu_1, \partial_t{\uu_2}) - j(\uu_1, \partial_t{\uu_1})  +   j(\uu_2, \partial_t{\uu_1}) - j(\uu_2, \partial_t{\uu_2})
  \\
  \leq \int_0^t \int_{\CS} c_T|(-[\uu_2]_{\CS} \cdot \boldsymbol n^{+} - g_0)_{+})^{m_T} - ([\uu_1]_{\CS} \cdot \boldsymbol n^{+} - g_0)_{+})^{m_T}|   |[\uu_1]_{T\CS} - [\uu_2]_{T\CS}| ds d\tau
  \\
  \leq C \int_0^t \|\nabla \uu_1 - \nabla \uu_2\| \|\nabla \partial_t\uu_1 - \nabla \partial_t\uu_2\|d\tau , 
\end{multline}
where $[\uu_i]_{T\CS}, i = 1,2$ denotes the tangential component of jump of $\uu_i$ on $\CS$. Here, we have used the restrictions on $m_n, m_T$ so that for $m_n+1, m_T+1$, $H^{1/2}(\CS)$ is continuously embedded in $L^{m_n+1}(\CS), L^{m_T+1}(\CS)$. Application of Young's inequality to absorb the $\nabla \partial_t$ term on the right hand side followed by Gronwall's lemma completes the proof. 
\subsection{Existence of solution}

We consider the existence of solution for the problem (Q). We begin with the regularization of the friction term. For the above variational problem (Q) the steps for well-posedness rely on the same strategy as in model problem (Q$_0$): Galerkin approximation for the finite dimensional projection, existence of the finite dimensional approximation, obtaining compactness estimates, extracting a subsequence that provides the limit quantities, and then showing that the limit quantities satisfy the continuous variational formulation \eqref{eq:varumixf1}-\eqref{eq:varpmixf3}. The relevant system of second order ODEs now contain nonlinear terms occurring due to the contact mechanics. Fortunately, the nonlinearity can be absorbed by the other terms while retaining the strict positivity of the coefficients appearing in the time derivative terms. Then the existence of a solution  result follows with appropriate identification of the limit quantities.

\subsubsection{Regularization of friction term} \label{sec:existenceFR}

We define 
\[
R(\boldsymbol u) = \dfrac{1}{m_n+1} \int_{\CS} c_n[(-[\uu]_\CS \cdot \boldsymbol n^{+} -g_0)_{+}]^{m_n+1} ds , \]
%and recall \begin{equation*}
%{j} (\uu, \vv) = \int_{\CS} c_T[([\uu]_{\CS} - g_0)]^{m_T}|\vv_T| ds, \quad \uu, \vv \in \V.   
%\end{equation*} 
which represents the energy associated with the normal deformation of the fracture interface. Following Martins \& Oden \cite{martins1987existence}, we define the regularization of $j$ by modifying $|\vv|$ factor in the friction term. The regularization ensures that the friction term becomes differentiable and hence we can replace the variational inequality with the equality. We define a family of convex functions $\psi_\ep \in C^1(\mathbb{R}^d,\mathbb{R})$ for $\ep >0$ satisfying the following conditions
\begin{enumerate}
      \item[(i)] $ 0 \leq {\psi_\ep} \leq |\vv|$,
      \item [(ii)] there exists $D_1>0$ such that the directional derivative $|{\psi_\ep^\prime}  (\boldsymbol w) (\vv) \leq D_1 |\vv|$,
    \item [(iii)]there exists $D_2>0$ such that $|{\psi_\ep} (\vv) - |\vv|| \leq {D_2} \ep$ uniform in $\ep$.
\end{enumerate}
We perform a regularization of $J$ by defining
\begin{equation}
{J_\ep} (\uu, \vv) = \int_{\CS} c_T[(-[\uu]_{\CS} \cdot \boldsymbol n^{+} - g_0)_{+}]^{m_T} {\psi_\ep} (\vv_T) ds, \quad \uu, \vv \in \V.   
\end{equation} 
The partial Gateaux derivative of $J_\ep$ with respect to the second argument of $J_\ep$ is then given by, for $\boldsymbol z \in \boldsymbol V$ 
\[\langle J_\epsilon (\boldsymbol w, \vv), \boldsymbol z \rangle = \int_{\CS} c_T [\left(-[\uu]_n \cdot \boldsymbol n^{+} - g_0 \right)_{+}]^{m_n} {\psi_\epsilon}^\prime (\vv_T) \boldsymbol z ds. \]
\subsubsection{Regularized problem (QR)}
The regularized problem is then given as: 
 The conformal variational formulation of Problem (QR) reads: For given $\gamma>0$, ${\boldsymbol f} \in L^2((\Omega \setminus \CS) \times ]0,T[)^d$, $\tilde q \in L^2(\Omega \times ]0,T[)$, and  $\tilde {q}_L^\ep \in L^2(0,T;H^{1}(\CS)')$, find  $\boldsymbol u^\ep \in L^\infty(0,T;\V)$, $\partial_t{\boldsymbol u}^\ep \in L^\infty(0,T; L^2(\Omega))  {\displaystyle \cap } L^2(0,T; \V)$, $\partial_{tt} \boldsymbol u^\ep \in L^2(0,T;\V^{'})$,  $p^\ep \in L^\infty(0,T;H^1(\Omega))$, $p_c^\ep \in L^2(0,T;H^{1}_w(\CS))$, such that 
 for all ${\boldsymbol v} \in \V$, 
\begin{multline}
 (\partial_{tt} {{\boldsymbol u}^\ep}, {\boldsymbol v}) + 2 G \big(\varepsilon({\boldsymbol u}^\ep), \varepsilon({\boldsymbol v})\big) + \lambda \big(\nabla \cdot {\boldsymbol u}^\ep,\nabla\cdot {\boldsymbol v}\big) + \gamma \big(\varepsilon(\partial_t{ \boldsymbol u}^\ep), \varepsilon({\boldsymbol v})\big) + \gamma \big(\nabla \cdot \partial_t{ \boldsymbol u}^\ep,\nabla\cdot {\boldsymbol v}\big)  - \alpha \big( p^\ep ,\nabla\cdot {\boldsymbol v}\big) \\
 +\big(p_c^\ep ,[{\boldsymbol v}]_{\CS} \cdot {\boldsymbol n}^+\big)_\CS + \langle P_n([\uu^\ep]_\CS), \vv \rangle +  \langle {J_\ep }(\uu(t), {\vv}) - {J_\ep} (\uu(t), \dot{\uu}^\ep(t) \rangle = \big( {\boldsymbol f} ,{\boldsymbol v}\big) 
\label{eq:varumixfr1}    
\end{multline} 
for all $\theta \in H^1(\Omega)$,
\begin{align}
\label{eq:varpmixfr2}    \Big({\partial_t}\Big(\Big(\frac{1}{M} + c_f \varphi_0\Big) p^\ep +\alpha  \nabla\cdot {\boldsymbol u}^\ep \Big), \theta\Big)
+ \frac{1}{\mu_f}\big( \boldsymbol K \nabla p^\ep,\nabla \theta \big) &= \big(\tilde q, \theta \big) + \langle \tilde{q}_L^\ep,\theta \rangle_{\CS}
\end{align}
for all $\theta_c \in H_w^1(\CS)$, and
\begin{multline}
(c_{fc} \partial_t p_c^\ep, \theta_c)_{\CS} +   (\partial_t (-[{\boldsymbol u}^\ep]_\CS \cdot {\boldsymbol n}^+), \theta_c )_\CS \\
+ \frac{1}{12 \mu_f} ( w^{\frac 3 2} (\overline{\nabla} p_c^\ep - \rho_{f,r} g \nabla \eta), w^{\frac 3 2} \overline{\nabla} \theta_c )_\CS =  ( \tilde q_W , \theta_c )_\CS,
- \langle \tilde{q}_L^\ep, \theta_c \rangle_\CS ,
\label{eq:varpmixfr3}
\end{multline}
 subject to the initial conditions,
\begin{align}
\label{eqn:initialfr1}
\quad \Big(\Big(\frac{1}{M} + c_f \varphi_0\Big)p^\ep + \alpha \nabla \cdot {\boldsymbol u}^\ep \Big)\Big|_{t=0} = \Big(\frac{1}{M} + c_f \varphi_0\Big)p_0 + \alpha \nabla \cdot {\boldsymbol u}_0
\end{align}
and $\uu(x,0) = \uu_0, \partial_t \uu(x,0) = \uu_1$.
We first approximate the solution in a finite dimensional space following the Galerkin approximation and before taking the limit $\ep \rightarrow 0$. 
\subsubsection{Existence of semi-discrete solution to the problem (QR)}
We consider a semi-discrete Galerkin scheme. Let $\{\theta_i\}_{i = 1}^{\infty}$ form a basis for $Q$, that is, for any $m\geq 1$, we approximate $Q$ by $Q_m$ having the basis $\{\theta_i, i\leq m\}$ and $Q = \overline{\bigcup_{m \geq 1} Q_m}$.
We define the Galerkin approximation  $p_m$  of $p^\ep$
\[p_m(t) = \sum_{j = 1}^m P_j(t) \theta_j,\]
with $p_j = H^2(0,T), 1\leq j\leq m$. 
Similarly, considering $\{\boldsymbol{w}_i\}_{i = 1}^\infty$ be a sequence of functions $\boldsymbol w_i \in V$ for all $i$ and let  $\boldsymbol w_i, i = 1, \ldots, m$ span the subspace $V_m$ of $V$ with $V = \overline{\bigcup_{m \geq 1} V_m}$. To approximate the solution $\uu_m: \Omega^+ \bigcup \Omega^- \times [0,T] \rightarrow \mathbb{R}^d, d = 2,3$ of mechanics equation \eqref{eq:varumixfr1}, we define 
\[\uu_m(\boldsymbol x,t) :=  \sum_{j=1}^m \boldsymbol U_j(t) \boldsymbol w_j(\boldsymbol x). \]
The solution of the fracture pressure $p_{cm}$ is obtained by taking the trace of $p_m$ on $\CS$.
For the leakage term, we have
\[\widetilde{q}_{m,L} (t) = \sum_{j=1}^m Q_j(t)\theta_{j,c}.\]
The coefficients of the approximations are chosen so that the model equations \eqref{eq:varumixfr1}-\eqref{eq:varpmixfr3} lead to the following system of coupled nonlinear ODEs:
\begin{multline}
\noindent \text{For } \boldsymbol{w}_i, i = 1, \ldots,m,\\
 (\partial_{tt} {{\boldsymbol u}_m}, {\boldsymbol w}_i) + 2 G \big(\varepsilon({\boldsymbol u}_m), \varepsilon({\boldsymbol w_i})\big) + \lambda \big(\nabla \cdot {\boldsymbol u}_m,\nabla\cdot {\boldsymbol w_i}\big) + \gamma (\dot{\uu}_m(t), \boldsymbol w_i) - \alpha \big( p_m ,\nabla\cdot {\boldsymbol w_i}\big) \\
 +\big(p_{cm} ,[{\boldsymbol w_i}]_{\CS} \cdot {\boldsymbol n}^+\big)_\CS + \langle P_n([\uu_m]_\CS), \boldsymbol w_i \rangle +  \langle {J_\ep }(\uu_m(t), {\boldsymbol w_i}) - {J_\ep} (\uu_m(t), \dot{\uu}_m(t) \rangle = \big( {\boldsymbol f} ,{\boldsymbol w_i}\big) 
\label{eq:varumixfrd1}    
\end{multline} 
for $\theta_i, i = 1,\ldots,m$, \\[0.4em]
\begin{align}
   \Big({\partial_t}\Big(\Big(\frac{1}{M} + c_f \varphi_0\Big) p_m +\alpha  \nabla\cdot {\boldsymbol u}_m \Big), \theta_i \Big)
+ \frac{1}{\mu_f}\big( \boldsymbol K \nabla p_m,\nabla \theta_i \big) &= \big(\tilde q, \theta_i \big) + \langle \tilde{q}_L,\theta_i \rangle_{\CS}, \label{eq:varpmixfrd2} 
\end{align}
and for $\theta_{i,c}, i=1, \ldots, m$,\\  
\begin{multline}
(c_{fc} \partial_t p_{cm}, \theta_{i,c})_{\CS} +   (-\partial_t[{\boldsymbol u}_m]_\CS \cdot {\boldsymbol n}^+, \theta_{i,c} )_\CS + \\
+ \frac{1}{12 \mu_f} ( w^{3/2} (\overline{\nabla} p_{cm} - \rho_{f,r} g \nabla \eta), w^{3/2} \overline{\nabla} \theta_{i,c} )_\CS =  ( \tilde q_W , \theta_{i,c} )_\CS
- \langle \tilde{q}_L, \theta_{i,c} \rangle_\CS.
\label{eq:varpmixfrd3}
\end{multline}
The initial conditions satisfy
\begin{align*}
    \uu^m(0) &= \uu_0^m = \sum_{i = 1}^m \alpha_i^m \boldsymbol{w}_i \rightarrow \uu_0 \text{ in } V \text{ as } m \rightarrow \infty,\\
    \partial_t \uu^m(0) &= \uu_1^m = \sum_{i = 1}^m \beta_i^m \boldsymbol{w}_i \rightarrow \uu_1 \text{ in } L^2(\Omega) \text{ as } m \rightarrow \infty,
\end{align*}
and analogously for the pressure $p_m(0)$.  This forms a system of ODEs and we can show the existence of a solution  by showing the existence for this system. \\
% For $1\leq i \leq m, \forall \theta_i \in H^1(\Omega)\,,\,$
% \begin{multline}
%  \Big(\big((\frac{1}{M} + c_f \varphi_0) \partial_t p_m +\alpha  \nabla\cdot {\partial_t \boldsymbol u_m} \big) ,\theta_i \Big)
%  + \frac{1}{\mu_f}\big( \boldsymbol K \nabla \partial_t p_m,\nabla \theta_i \big) +(c_{fc}  \partial_t p_{cm}, \theta_i)_\CS \\
%   + ([{\partial_t \boldsymbol u}]_\CS \cdot {\boldsymbol n}^+, \theta_i)_\CS + \frac{1}{12 \mu_f} \big ( w^{\frac{3}{2}} (\overline{\nabla} p_{cm} - \rho_{f,r} g \eta), w^{\frac{3}{2}} \overline{\nabla} \theta_i \big )_\CS 
%  = \big(\tilde q, \theta \big),
% \label{eq:varp4}
% \end{multline}
We define the following matrices

\begin{tabular}{ l l l}
     \hspace{-1cm} $A^{up}_{ij} =  \int_{\Omega^{+} \cup \Omega^{-}} \theta_j \nabla \cdot \boldsymbol w_i dx$, & $C_{\uu ij} = \int_\Omega  \boldsymbol w_i \boldsymbol w_j dx$,  & \hspace{-2cm} $B_{ij} = \int_{\Omega^{+} \cup \Omega^{-}} \alpha \left(\nabla \cdot {\boldsymbol w}_j \right) \left (\nabla \cdot {\boldsymbol w}_i \right) dx$\\[1em]
      \hspace{-1cm}  $C_{p ij} = \int_\Omega \Big(\frac{1}{M}+c_f \varphi_0 \Big) \theta_i \theta_j dx$, &  $C_{pc ij} = \int_\CS c_{fc}   \theta_{i,c} \theta_{j,c} ds$, & \hspace{-2cm} $L_{ij} = \int_{\Omega^{+} \cup \Omega^{-}} \frac{1}{\mu_f} {\boldsymbol K}\nabla \theta_j \cdot \nabla \theta_i dx$, \\[1em]
      \hspace{-1cm}  $M_{ij}^{upc} = -\int_{\CS}[\boldsymbol w_j]_{|\CS} \cdot {\boldsymbol n}^{+} \theta_{i,c} ds$, & $W_{ij} = \int_{\CS} \dfrac{w^3}{12 \mu_f} \bar \nabla \theta_{j,c} \cdot  \bar \nabla \theta_{i,c} ds$, &  {} \\[1em]
      \hspace{-1cm}  $E_{1ij} = \int_{\Omega^{+} \cup \Omega^{-}} 2G {\boldsymbol \varepsilon}({\boldsymbol w}_j) : {\boldsymbol \varepsilon}({\boldsymbol w}_i) dx$, & $E_{2ij} = \int_{\Omega^{+} \cup \Omega^{-}} \lambda ({\nabla \cdot} {\boldsymbol w}_j) ({\nabla \cdot} {\boldsymbol w}_i)dx$. & {} \\[1em] 
\end{tabular}

% \begin{array}{lll}
% $
%  A^{up}_{ij} =  \int_\Omega \theta_j \nabla \cdot \boldsymbol w_i dx, \quad
% A^{upc}_{ij} =  \int_\Omega \theta_j \nabla \cdot \boldsymbol w_i dx, \quad
% B_{ij} = \int_{\Omega^{+} \cup \Omega^{-}} \alpha \left(\nabla \cdot {\boldsymbol w}_j \right) \left (\nabla \cdot {\boldsymbol w}_i \right) dx\\
%  C_{p ij} = \int_\Omega  \theta_i \theta_j dx,  \quad 
% C_{\uu ij} = \int_\Omega  \boldsymbol w_i \boldsymbol w_j dx, \quad L_{ij} = \int_\Omega \frac{1}{\mu_f} {\boldsymbol K}\nabla \theta_j \cdot \nabla \theta_i dx, \\
% M_{ij} = -\int_{\CS}[\boldsymbol w_j]_{|\CS} \cdot {\boldsymbol n}^{+} \theta_{i,c} ds,  \quad 
% W_{ij} = \int_{\CS} \dfrac{w^3}{12 \mu_f} \bar \nabla \theta_{j,c} \cdot  \bar \nabla \theta_{i,c} ds. \\
%  E_{ij} = \int_{\Omega^{+} \cup \Omega^{-}} 2G {\boldsymbol \varepsilon}({\boldsymbol w}_j) : {\boldsymbol \varepsilon}({\boldsymbol w}_i) dx + \int_{\Omega^{+} \cup \Omega^{-}} \lambda ({\nabla \cdot} {\boldsymbol w}_j) ({\nabla \cdot} {\boldsymbol w}_i)dx. 
% $
% \end{array}
We denote $U= [\boldsymbol U_1, \boldsymbol U_2, \ldots, \boldsymbol U_m]^T$; then the elasticity equation results in the equation
\begin{align}
    C_\uu \dfrac{d^2 {\boldsymbol U}}{dt^2} + (E_1 + E_2)  {\boldsymbol U} + \gamma E_1  \dfrac{d \boldsymbol U} {dt} + R({\boldsymbol U})+J({\boldsymbol U}) + A^{up} P = F.
\end{align}
We denote $P= [p_1, p_2, \ldots, p_m]^T$; then the mass balance equations in fracture and matrix result in the equation
\begin{align}
 C_p \dfrac{d P}{dt} + A^{up} \dfrac{d U}{dt}  + LP + WP_c +  C_{pc}\dfrac{d P_c}{dt} = Q.    
\end{align}
% Using the elasticity equation, the above ODE system takes the form
% \[C\dfrac{d^2 P}{dt^2} + \Big(\Big(\frac{1}{M}+c_f \varphi_0 \Big) C + E\Big)\dfrac{d P}{dt}  + (L+W)P = Q. \]
% This is a second order non-linear system of ODEs with identity in the highest order, $(C+E)$ matrix as coefficients being symmetric positive definite. 
To assert the unique solution for the system of ODEs, we write down a fixed point scheme. Together with the initial conditions the convergence of the fixed point provides solution to the above system. 
The fixed point scheme is given by: Start with $\boldsymbol U^0, P^0, P_c^0$ as the initial condition, given $\boldsymbol U^{k-1}, P^{k-1}, P_c^{k-1}$, solve for $\boldsymbol U^{k}, P^{k}, P_c^{k}$ satisfying
\begin{align}
\nonumber   C_\uu \dfrac{d^2 {\boldsymbol U^k}}{dt^2} + (E_1 + E_2)  {\boldsymbol U^k} + \gamma E_1  \dfrac{d \boldsymbol U^k} {dt} + R({\boldsymbol U^{k-1}})+J({\boldsymbol U^{k-1}}) + A^{up} P^{k} = F. \\
\label{eq:fixedpoint}     C_p \dfrac{d P^k}{dt} + A^{up} \dfrac{d U^k} {dt}  + LP^k + WP_c^k +  C_{pc}\dfrac{d P_c^k}{dt} = Q.  
 \end{align}
Note that the above system is a linear system as the nonlinearity has been iteration-lagged. The existence and uniqueness of the linear system has already been achieved in the previous section. Following the estimates in the proof for uniqueness of solution for problem (Q), contraction of the above scheme follows and thereby we obtain the proof of existence. We spare the details. 

\begin{remark}
The above fixed point scheme \eqref{eq:fixedpoint} decouples the above system in a linear and nonlinear parts which are iteration lagged. The linear part contains coupled terms of flow in fracture and matrix and mechanics however decouples the friction and normal compliance terms. This is however not the typical approach for solving this system. This coupled problem is solved using fixed stress algorithm \cite{girault2016convergence,storvik2019optimization} which decouples the flow and the mechanics models. This approach adds additional stabilization to the pressure equation to ensure the unconditional convergence. Similarly, there are other iterative approaches including the  fixed strain algorithm  and undrained splitting \cite{kim11,kim2, almani_undrained, ref:Mikwheel12} that decouple the flow and mechanics equations. \par
To solve numerically, we also need to consider temporal discretization of the above scheme. Since the above model is a multiphysics problem, it is advantageous to have different time stepping schemes for different physical effects. Accordingly, multirate schemes have been developed that exploit the difference in characteristic time scales while maintaining accuracy and convergence of the numerical scheme. Such works have been considered in \cite{almani_frac_22, almani_fixedstress, almani_frac_19,  almani_cma_19}.
\end{remark}

\subsubsection{A priori estimates to the problem (QR)}
We proceed to derive estimates using the approximate solutions as the test functions. This is allowed as we are in conformal situation. We begin with the variational form \eqref{eq:varumixf1}-\eqref{eq:varpmixf3} and choose $\boldsymbol v = \partial_t {\boldsymbol u_m^\ep}, \theta =  p_m^\ep, \theta_c = p_{cm}^\ep$, to obtain
\begin{multline}  \label{eq:a?p} 
 (\partial_{tt} {\boldsymbol u_m^\ep}, \partial_t{\boldsymbol u_m^\ep})
+ 2 G \big(\varepsilon({\boldsymbol u_m^\ep}), \varepsilon({\partial_t \boldsymbol u_m^\ep})\big) + \lambda \big(\nabla \cdot { \boldsymbol u_m^\ep},\nabla\cdot {\partial_t \boldsymbol u_m^\ep}\big)  
+ \gamma \big(\varepsilon( \partial_t{\boldsymbol u_m^\ep}), \varepsilon({\partial_t \boldsymbol u_m^\ep})\big) \\[0.15cm]
+ \gamma \big(\nabla \cdot \partial_t{ \boldsymbol u_m^\ep},\nabla\cdot {\partial_t \boldsymbol u_m^\ep}\big) 
- \alpha \big( p_m^\ep ,\nabla\cdot {\partial_t \boldsymbol u_m^\ep}\big) 
 +\big(p_{cm}^\ep ,[{\partial_t \boldsymbol u_m^\ep}]_{\CS} \cdot {\boldsymbol n}^+\big)_\CS 
+  \partial_t R(\uu_m^\ep) + (c_{fc} \partial_t p_{cm}^\ep, p_{cm}^\ep) \\[0.15cm]
+ \langle J_\ep(\uu_m^\epsilon, \partial_t \uu_m^\ep), \partial_t \uu_m^\ep \rangle  + \Big({\partial_t}\Big(\Big(\frac{1}{M} + c_f \varphi_0\Big) p_m^\ep  +\alpha  \nabla\cdot {\boldsymbol u_m^\ep} \Big) , p_m^\ep \Big)
  + \frac{1}{\mu_f}\big( \boldsymbol K \nabla p_m^\ep,\nabla p_m^\ep \big) \\
 -  (\partial_t[{\boldsymbol u_m^\ep}]_\CS \cdot {\boldsymbol n}^+, p_{cm}^\ep )_\CS 
 + \frac{1}{12 \mu_f} \langle w^{\frac{3}{2}} (\overline{\nabla} p_{cm}^\ep - \rho_{f,r} g \eta), w^{\frac{3}{2}} \overline{\nabla} p_{cm}^\ep \rangle = \langle\tilde q_W , p_{cm}^\ep \rangle_\CS + \big(\tilde q, p_m^\ep \big) + \big( {\boldsymbol f} ,{\boldsymbol u_m^\ep}\big). \end{multline} 
The only change from the previous analogous estimate in the linear case  is the appearance of the term $R(\uu_m^\ep)$ and $\langle J_\ep(\uu_m^\ep, \partial_t \uu_m^\ep), \partial_t \uu_m^\ep \rangle$. Both these terms are nonnegative and, hence, the following estimate follows 
\begin{proposition} \label{pro:estimatesQ3R}
Let the data satisfy: $\boldsymbol f \in H^2(0,T; L^2(\Omega)^d)$, $\tilde{q} \in L^2(\Omega \times (0,T))$, $\tilde{q}_W \in H^1(0,T;L^2(\CS))$, $p_0 \in Q$, and $w$ satisfies \eqref{eq:hyp_w}, then the solution $\uu_m^\ep, p_m^\ep, p_{cm}^\ep$ satisfying  \eqref{eq:varumixf1}-\eqref{eq:varpmixf3} satisfy the following estimates
 \begin{multline} \label{eq:aprioriestimatesQ3R} 
 \|\partial_t \uu_m^\ep(t)\|^2 + G \|\varepsilon(\uu_m^\ep)(t) \|^2 
+ \lambda \|\nabla \cdot \uu_m^\ep(t) \|^2  
+ \Big(\frac{1}{M} + c_f \varphi_0\Big) \| p_m^\ep(t) \|^2 +  \gamma \|\nabla \partial_t \uu_m^\ep(t) \|^2  \\
+ c_{fc} \|p_{cm}^\ep(t) \|^2 + \frac{1}{\mu_f} \int_0^t \| \boldsymbol K^{1/2} \nabla p_m^\ep(s) \|^2 ds
+ \frac{1}{12 \mu_f} \int_0^t \| w^{\frac{3}{2}} \overline{\nabla} p_{cm}^\ep(s) \|^2 ds 
 \leq C .
   \end{multline}
Here, $C$ depends on the data and is independent of $m, \ep$ and the solution triple $(\uu_m^\ep, p_m^\ep, p_{cm}^\ep)$.
\end{proposition}
We will keep $\ep$ fixed and take the limit $m \rightarrow \infty$. The above estimate implies the following convergence result, if needed, up to a subsequence,  $t \le T$
\begin{align*}
 p_m^\ep &\rightarrow p^\ep \text{ weakly in } L^\infty(0,t; L^2(\Omega)) \cap L^2(0,t; Q)\\  
 %p_{cm} &\rightarrow p_c \text{ weakly in } L^\infty(0,t; H^1(\CS))\\
 \uu_m^\ep & \rightarrow \uu^\ep \text{ weak-star in } L^\infty(0,t; \V) \\
 \partial_t \uu_m^\ep & \rightarrow \partial_t \uu^\ep \text{ weak-star in }    L^\infty(0,t; L^2(\Omega)) \text{ and weakly in }  L^2(0,t; \V).
\end{align*}
To pass to the limit, we observe that
\[c_n[(-[\uu_m^\ep]_\CS \cdot \boldsymbol n^{+} - g_0)_{+}]^{m_n} \in \text{ bounded set of  } L^\infty(0,t;L^{q^*}(\CS)), \]
where $L^{q^*}(\CS)$ is the dual norm of $L^{m_n+1}(\CS)$. This yields,
\[c_n[(-[\uu^\ep_m]_\CS \cdot \boldsymbol n^{+} - g_0)_{+}]^{m_n} \rightarrow \chi \text{ weak-star in } L^\infty(0,t;L^{q^*}(\CS)). \]
In order to show that $\chi = c_n[(-[\uu^\ep]_\CS \cdot \boldsymbol n^{+} - g_0)_{+}]^{m_n}$, we observe that 
$\uu_m^\ep$ is bounded in $(H^1(\Omega \times (0,t)))^d$. Using the compactness of the trace map from $(H^1(\Omega \times (0,t)))$ to $(L^2(\CS \times (0,t)))$, it follows that  $t \le T$
\[\uu_m^\ep \rightarrow \uu^\ep \text{ strongly in }  (L^2(\CS \times (0,t)))^d\] and
\[\uu_m^\ep \rightarrow \uu \text{ a.e. on }  \CS \times (0,t).\]
This leads to 
\[c_n[(-[\uu_m^\ep]_\CS \cdot \boldsymbol n^{+} - g_0)_{+}]^{m_n} \rightarrow c_n[(-[\uu^\ep]_\CS \cdot \boldsymbol n^{+} - g_0)_{+}]^{m_n} \text{ a.e. on } \CS \times (0,t). \]
Moreover, the boundedness of $c_n[(-[\uu_m^\ep]_\CS \cdot \boldsymbol n^{+} - g_0)_{+}]^{m_n}$ in $L^{q^*}(0,t;L^{q^*}(\CS))$ norm  means that it converges to some $\chi$ weakly in this norm. The uniqueness of the limit identifies 
$\chi = c_n[(-[\uu^\ep]_\CS \cdot \boldsymbol n^{+} - g_0)_{+}]^{m_n}$. Summarizing,  we have the following limits,
\begin{align*}
c_n[(-[\uu_m^\ep]_\CS \cdot \boldsymbol n^{+} - g_0)_{+}]^{m_n} \rightarrow c_n[(-[\uu^\ep]_\CS \cdot \boldsymbol n^{+}- g_0)_{+}]^{m_n} \text{ weak-star in } L^\infty(0,T;L^{q^*}(\CS)), \\      
c_T[(-[\uu_m^\ep]_\CS \cdot \boldsymbol n^{+} - g_0)_{+}]^{m_T} \rightarrow c_T[(-[\uu^\ep]_\CS \cdot \boldsymbol n^{+} - g_0)_{+}]^{m_T} \text{ weak-star in } L^\infty(0,T;L^{q^*}(\CS)). 
\end{align*}
The above weak convergence can be improved to strong convergence in $L^\infty(0,T;L^{q^*}(\CS))$. This follows from the compact injection of $\{\vv \in H^{1/2}(\partial \Omega): \vv =0 \text{ a.e. on } \Gamma_D\}$ in $\{\vv \in L^{q}(\partial \Omega): \vv =0 \text{ a.e. on } \partial \Omega\}$ leading to $\uu_m^\ep \rightarrow  \uu^\ep$ in $L^2(0,T; L^q(\CS))$, which leads to 
\begin{align*}
c_n[(-[\uu_m^\ep]_\CS \cdot \boldsymbol n^{+} - g_0)_{+}]^{m_n} &\rightarrow c_n[(-[\uu^\ep]_\CS \cdot \boldsymbol n^{+} - g_0)_{+}]^{m_n} \text{ strongly  in } L^\infty(0,T;L^{q^*}(\CS)), \\      
c_T[(-[\uu_m^\ep]_\CS \cdot \boldsymbol n^{+} - g_0)_{+}]^{m_T} &\rightarrow c_T[(-[\uu^\ep]_\CS \cdot \boldsymbol n^{+} - g_0)_{+}]^{m_T} \text{ strongly  in } L^\infty(0,T;L^{q^*}(\CS)). 
\end{align*}
It accordingly follows that
\begin{align*}
2G\big(\varepsilon({\boldsymbol u_m^\ep}), \varepsilon({ \vv_j})\big) + \lambda \big(\nabla \cdot { \boldsymbol u_m^\ep},\nabla\cdot { \vv_j}\big) \rightarrow 2G\big(\varepsilon({\boldsymbol u^\ep}), \varepsilon({ \vv_j})\big) + \lambda \big(\nabla \cdot { \boldsymbol u^\ep},\nabla\cdot { \vv_j}\big) &\text{ weak-star in } L^\infty(0,T),\\ 
\big(\varepsilon({\partial_t \boldsymbol u_m^\ep}), \varepsilon({ \vv_j})\big) +  \big(\nabla \cdot { \partial_t \boldsymbol u_m^\ep},\nabla\cdot { \vv_j}\big) \rightarrow \big(\varepsilon(\partial_t {\boldsymbol u^\ep}), \varepsilon({ \vv_j})\big) +  \big(\nabla \cdot { \partial_t \boldsymbol u^\ep},\nabla\cdot { \vv_j}\big) &\text{ weak-star in } L^\infty(0,T),\\ 
(\partial_{tt} {\boldsymbol u_m^\ep}, \vv_j) \rightarrow (\partial_{tt} {\boldsymbol u^\ep}, \vv_j) &\text{ weak-star in } \mathcal{D}^\prime(0,T),\\
 \big( p_m^\ep ,\nabla\cdot \vv_j\big) \rightarrow  \big( p^\ep ,\nabla\cdot \vv_j\big) &\text{ weak-star in } L^\infty(0,T),\\
 \big(p_{cm}^\ep ,[{\partial_t \boldsymbol v_m^\ep}]_{\CS} \cdot {\boldsymbol n}^+\big)_\CS \rightarrow 
 \big(p_{c}^\ep ,[{\partial_t \boldsymbol v^\ep}]_{\CS} \cdot {\boldsymbol n}^+\big)_\CS &\text{ weak-star in } L^\infty(0,T),\\
 \frac{1}{\mu_f}\big( \boldsymbol K \nabla p_m^\ep,\nabla \theta_j \big) \rightarrow \frac{1}{\mu_f}\big( \boldsymbol K \nabla p^\ep,\nabla \theta_j \big) &\text{ weakly in } L^2(0,T),\\
% + \dfrac{d}{dt} R(\uu_m^\ep) + \langle J_\ep(\uu_\epsilon, \dot \uu), \dot \uu \rangle  + \Big(\frac{\partial }{\partial t}\big((\frac{1}{M} + c_f \varphi_0) p_m^\ep  +\alpha  \nabla\cdot {\boldsymbol u_m^\ep} \big) , p_m^\ep \Big)\\
  (\partial_t{\uu}_m^\ep , \vv_j) \rightarrow (\partial_t{\uu}^\ep, \vv_j) &\text{ weak-star in } L^\infty(0,T),\\
 \big(p_{cm}^\ep ,[{\boldsymbol v}_j]_{\CS} \cdot {\boldsymbol n}^+\big)_\CS \rightarrow   \big(p_{c}^\ep ,[{\boldsymbol v}_j]_{\CS} \cdot {\boldsymbol n}^+\big)_\CS &\text{ weak-star in } L^\infty(0,T),\\
 ( w^{3/2} (\overline{\nabla} p_{cm}^\ep - \rho_{f,r} g \nabla \eta), w^{3/2} \overline{\nabla} \theta_{cj} )_\CS \rightarrow  ( w^{3/2} (\overline{\nabla} p_c^\ep - \rho_{f,r} g \nabla \eta), w^{3/2} \overline{\nabla} \theta_{cj} )_\CS &\text{ weakly in } L^2(0,T),\\
   \langle P_n([\uu_m^\ep]_\CS), \vv_j \rangle \rightarrow \langle P_n([\uu^\ep]_\CS), \vv_j \rangle &\text{ weak-star in } L^\infty(0,T),\\
  \langle {J_\ep }(\uu_m^\ep, {\vv}_j) \rightarrow \langle \chi^\ep, \vv_j) \rangle &\text{ weak-star in } L^\infty(0,T).
\end{align*}
The identification of $\chi^\ep$ with  ${J_\ep }(\uu_m^\ep, \partial_t{\vv}_j)$ takes place by using monotonicity arguments. We note that the monotonicity of $J_\ep$ in its second argument implies \[ 0 \leq X^m { :=} \langle J_\ep(\uu_m^\ep), \partial_t \uu^\ep_m \rangle -  J_\ep(\uu_m^\ep, \phi), \partial_t \uu^\ep_m - \phi \rangle \text{ for all } \phi \in L^2(0,T; \V). \]
Using the above convergence result, we obtain the following equations in all the terms that are linear. The two nonlinear terms are in $P(\uu_m^\ep)$ and $J_\ep(\uu_m^\ep, \partial_t \uu_m^\ep)$. The first convergence has already been dealt with in the section before. This leads to the following model equations in variational form:
For all $ {\boldsymbol v} \in \V,\;$
\begin{multline}
 (\partial_{tt} {{\boldsymbol u}^\ep}, {\boldsymbol v}) + 2 G \big(\varepsilon({\boldsymbol u}^\ep), \varepsilon({\boldsymbol v})\big) + \lambda \big(\nabla \cdot {\boldsymbol u}^\ep,\nabla\cdot {\boldsymbol v}\big) + \gamma (\partial_t{\uu}^\ep, \vv) - \alpha \big( p^\ep ,\nabla\cdot {\boldsymbol v}\big) \\
 +\big(p_c^\ep ,[{\boldsymbol v}]_{\CS} \cdot {\boldsymbol n}^+\big)_\CS + \langle P_n([\uu^\ep]_\CS), \vv \rangle +  \langle {\chi^\ep }, \vv \rangle = \big( {\boldsymbol f} ,{\boldsymbol v}\big),
\end{multline} 
for all $\theta \in H^1(\Omega)$.
\begin{align}
  \Big({\partial_t}\Big(\Big(\frac{1}{M} + c_f \varphi_0\Big) p^\ep +\alpha  \nabla\cdot {\boldsymbol u}^\ep \Big), \theta\Big)
+ \frac{1}{\mu_f}\big( K \nabla p^\ep,\nabla \theta \big) &= \big(\tilde q, \theta \big) + \langle \tilde{q}_L^\ep,\theta \rangle_{\CS},
\end{align}
and for all $\theta_c \in H_w^1(\CS)$,
\begin{multline}
(c_{fc} \partial_t p_c^\ep, \theta_c)_{\CS} -  (\partial_t[{\boldsymbol u}^\ep]_\CS \cdot {\boldsymbol n}^+, \theta_c )_\CS
\\
+ \frac{1}{12 \mu_f} ( w^{3/2} (\overline{\nabla} p_c^\ep - \rho_{f,r} g \nabla \eta), w^{3/2} \overline{\nabla} \theta_c )_\CS = ( \tilde q_W , \theta_c )_\CS,
- \langle \tilde{q}_L^\ep, \theta_c \rangle_\CS.
\end{multline}
 The above compactness properties are sufficient to guarantee that
\[\langle  P_n(\uu_m^\ep), \vv \rangle \rightarrow \langle  P_n(\uu^\ep), \vv \rangle \text{ weak-star in } L^\infty(0,t). \]
It remains to be shown that $J_\ep$ converges to $\chi^\ep = J_\ep(\uu^\ep, \partial_t \uu^\ep)$.
Integrating over time from $0$ to $T$ the weak formulation after testing with the triple $(\uu_m^\ep, p_m^\ep, p_{cm}^\ep)$, we obtain
\begin{multline*}
 ( \partial_t{\boldsymbol u_m^\ep}, \partial_t{\boldsymbol u_m^\ep})
+ 2 G \big(\varepsilon({\boldsymbol u_m^\ep}), \varepsilon({\partial_t \boldsymbol u_m^\ep})\big) + \lambda \big(\nabla \cdot { \boldsymbol u_m^\ep},\nabla\cdot {\boldsymbol u_m^\ep}\big) 
\nonumber +  R(\uu_m^\ep) + \langle J_\ep(\uu_m^\ep, \partial_t \uu_m^\ep), \partial_t \uu_m^\ep \rangle \\[0.15cm]
+ \Big(\Big(\frac{1}{M} + c_f \varphi_0\Big) p_m^\ep , p_m^\ep \Big)
+ \frac{1}{\mu_f}\big( K \nabla p_m^\ep,\nabla p_m^\ep \big) 
+ (c_{fc}  p_{cm}^\ep, p_{cm}^\ep) 
\\
 + \frac{1}{12 \mu_f} \langle w^{\frac{3}{2}} (\overline{\nabla} p_{cm}^\ep), w^{\frac{3}{2}} \overline{\nabla} p_{cm}^\ep \rangle
 =  \langle\tilde q_W , p_{cm}^\ep \rangle_\CS + \big(\tilde q, p_m^\ep \big) + \big( {\boldsymbol f} ,\partial_t{\boldsymbol u_m^\ep}\big)   \end{multline*} 
so that 
\begin{multline*}
  0 \leq X^m = - \dfrac{1}{2}\|\partial_t \uu_m^\ep(T)\|^2 - 
  G \|\varepsilon(\uu_m^\ep)(T) \|^2 
- \lambda \dfrac{1}{2}\|\nabla \cdot \uu_m^\ep(T) \|^2  
- \Big(\frac{1}{M} + c_f \varphi_0\Big) \dfrac{1}{2}\| p_m^\ep(T) \|^2 \\
 - c_{fc} \dfrac{1}{2}\| p_{cm}^\ep(T) \|^2 - \frac{1}{\mu_f} \int_0^t \| K^{1/2} \nabla p_m^\ep(s) \|^2 ds - R(\uu_m^\ep(T))  
- \frac{1}{12 \mu_f} \int_0^t \| w^{\frac{3}{2}} \overline{\nabla} p_{cm}^\ep(s) \|^2 ds \\
+ \dfrac{1}{2}\|\partial_t \uu_m^\ep(0)\|^2 + 
  G \|\varepsilon(\uu_m^\ep)(0) \|^2 
+ \lambda \dfrac{1}{2}\|\nabla \cdot \uu_m^\ep(0) \|^2  
+ \Big(\frac{1}{M} + c_f \varphi_0\Big) \dfrac{1}{2}\| p_m^\ep(0) \|^2 \\
 + c_{fc} \dfrac{1}{2}\| p_{cm}^\ep(0) \|^2 + R(\uu_m^\ep(0)) + \int_0^T(\boldsymbol f, \partial_t \uu_m^\ep) - \gamma(\partial_t \uu_m^\ep, \partial_t \uu_m^\ep ) dt \\
 -\int_0^T \langle J_\ep(\uu_m^\ep, \partial_t \uu_m^\ep), \phi \rangle dt - \int_0^T \langle J_\ep(\uu_m^\ep, \phi), \partial_t \uu_\ep^m -\phi \rangle dt \text{ for all } \phi \in L^2(0,T; \V).
 \end{multline*}
With the convergence properties and the weak lower semi-continuity of the various terms, we obtain
\begin{multline*}
  0 \leq \lim \text{ sup }_{ m\rightarrow \infty}   X^m \leq \\[0.15cm]
  - \dfrac{1}{2}\|\partial_t \uu^\ep(T)\|^2 - 
  G \|\varepsilon(\uu^\ep)(T) \|^2 
- \lambda \dfrac{1}{2}\|\nabla \cdot \uu^\ep(T) \|^2  
- \Big(\frac{1}{M} + c_f \varphi_0\Big) \dfrac{1}{2}\| p^\ep(T) \|^2 \\
 - c_{fc} \dfrac{1}{2}\| p_{c}^\ep(T) \|^2 - \frac{1}{\mu_f} \int_0^t \| K^{1/2} \nabla p^\ep(s) \|^2 ds - R(\uu^\ep(T))  
- \frac{1}{12 \mu_f} \int_0^t \| w^{\frac{3}{2}} \overline{\nabla} p_{c}^\ep(s) \|^2 ds \\
+ \dfrac{1}{2}\|\partial_t \uu^\ep(0)\|^2 + 
  G \|\varepsilon(\uu^\ep)(0) \|^2 
+ \lambda \dfrac{1}{2}\|\nabla \cdot \uu^\ep(0) \|^2  
+ \Big(\frac{1}{M} + c_f \varphi_0\Big) \dfrac{1}{2}\| p^\ep(0) \|^2 \\
 + c_{fc} \dfrac{1}{2}\| p_{c}^\ep(0) \|^2 + R(\uu^\ep(0)) + \int_0^T(\boldsymbol f, \partial_t \uu^\ep) - \gamma(\partial_t \uu^\ep, \partial_t \uu^\ep ) dt \\
 -\int_0^T \langle \chi^\ep, \phi \rangle dt - \int_0^T \langle J_\ep(\uu^\ep, \phi), \partial_t \uu_\ep -\phi \rangle dt \text{ for all } \phi \in L^2(0,T; \V).
 \end{multline*}
Comparing the right-hand side against the variational weak form (84)-(86) upon testing with $(\uu_m^\ep, p_m^\ep, p_{cm}^\ep)$, we observe that, for all $\phi \in L^2(0,T; \V)$,
\begin{align*}
    0 \leq \int_0^T \langle \chi^\ep -  J_\ep(\uu_m^\ep, \phi), \partial_t \uu_\ep^m -\phi \rangle dt.
\end{align*}
We thus conclude that $\chi^\ep = J_\ep(\uu^\ep, \partial_t \uu_\ep) \text{ in } L^2(0,T; V')$.

Next, we take $\ep \rightarrow 0$. Noting that the estimates are independent of $\ep$, both the convergence and the identification of the limit are similar to the just discussed limiting procedure. We omit the details. 
%The limit equations follow by using smooth test functions and identifying the equations that the limit quantities satisfy. 

\section{Discussion}

We have developed and analyzed a model for flow and deformations in a fractured subsurface medium. This development required a coupling of flow, mechanical effects including accounting for the friction, contact forces and flow in the fracture. The model includes the effect of fluid pressure on the contact forces that depend on the penetration depth between the fracture surfaces. The friction is represented by a modification of Coulomb's law that supports a variational inequality. We established well-posedness using the tools of a priori estimates and passing to the limit in the framework of Martins \& Oden \cite{martins1987existence} and earlier works \cite{girault2015lubrication, girault2019mixed}.
\par Beyond the current model, it is natural to consider two extensions: Accounting for permeability dependence on the stress, and for friction behavior in more general fault dynamics. The rate-and-state dependent friction laws are widely applied in dynamic fault models and critically influence a fault's possible slip features. Furthermore, we limited our analysis to a single fracture whereas in reality a fracture network may appear where interactions between fractures come into play. These extensions will be part of future work.

\section*{Acknowledgement}

MVdH was supported by the Simons Foundation under the MATH + X program, the National Science Foundation under grant DMS-2108175, the Department of Energy under grant DE-SC0020345, and the corporate members of the Geo-Mathematical Imaging Group at Rice University. KK was supported by the Research Council of Norway and Equinor ASA through grant number 267908. KK also acknowledges funding from the VISTA programme, The Norwegian Academy of Science and Letters.
% .. triggering mechanisms, seismicity -- pore pressure -- tidal forcing -- ..

\bibliographystyle{siam} 
\bibliography{Rupture} 
\end{document}